\documentclass[12 pt,oneside,reqno]{amsart}

\usepackage[a4paper, total={6in, 8in}]{geometry}
\usepackage[dvips]{graphicx}
\usepackage{calc}
\usepackage{color}
\usepackage{amsmath}
\usepackage{amssymb}
\usepackage{amscd}
\usepackage{amsthm}
\usepackage{amsbsy}
\usepackage{delarray}
\usepackage{enumerate}
\usepackage[T1]{fontenc}
\usepackage{mathrsfs,amsmath}
\usepackage{inputenc}
\usepackage{enumerate}
\usepackage{hyperref}

\begin{document}



\setlength{\parindent}{5mm}
\renewcommand{\leq}{\leqslant}
\renewcommand{\geq}{\geqslant}
\newcommand{\N}{\mathbb{N}}
\newcommand{\sph}{\mathbb{S}}
\newcommand{\Z}{\mathbb{Z}}
\newcommand{\R}{\mathbb{R}}
\newcommand{\C}{\mathbb{C}}
\newcommand{\F}{\mathbb{F}}
\newcommand{\g}{\mathfrak{g}}
\newcommand{\h}{\mathfrak{h}}
\newcommand{\K}{\mathbb{K}}
\newcommand{\RN}{\mathbb{R}^{2n}}
\newcommand{\ci}{c^{\infty}}
\newcommand{\derive}[2]{\frac{\partial{#1}}{\partial{#2}}}
\renewcommand{\S}{\mathbb{S}}
\renewcommand{\H}{\mathbb{H}}
\newcommand{\eps}{\varepsilon}
\newcommand{\propo}{Proposition}
\newcommand{\lag}{Lagrangian}
\newcommand{\diffeo}{diffeomorphism}
\newcommand{\hamil}{Hamiltonian}
\newcommand{\sub}{submanifold}
\newcommand{\qmor}{homogeneous quasimorphism}
\newcommand{\enum}{enumerate}
\newcommand{\wt}{\widetilde}

\theoremstyle{plain}
\newtheorem{theo}{Theorem}
\newtheorem{prop}[theo]{Proposition}
\newtheorem{lemma}[theo]{Lemma}
\newtheorem{definition}[theo]{Definition}
\newtheorem*{notation*}{Notation}
\newtheorem*{notations*}{Notations}
\newtheorem{corol}[theo]{Corollary}
\newtheorem{conj}[theo]{Conjecture}
\newtheorem{question}[theo]{Question}
\newtheorem{claim}[theo]{Claim}

\newenvironment{demo}[1][]{\addvspace{8mm} \emph{Proof #1.
    ~~}}{~~~$\Box$\bigskip}

\newlength{\espaceavantspecialthm}
\newlength{\espaceapresspecialthm}
\setlength{\espaceavantspecialthm}{\topsep} \setlength{\espaceapresspecialthm}{\topsep}

\newenvironment{example}[1][]{
\vskip \espaceavantspecialthm \noindent \textsc{Example
#1.} }%
{\vskip \espaceapresspecialthm}

\newenvironment{remark}[1][]{\refstepcounter{theo} 
\vskip \espaceavantspecialthm \noindent \textsc{Remark~\thetheo
#1.} }%
{\vskip \espaceapresspecialthm}

\def\bb#1{\mathbb{#1}} \def\m#1{\mathcal{#1}}

\def\momeg{(M,\omega)}
\def\co{\colon\thinspace}
\def\Homeo{\mathrm{Homeo}}
\def\Diffeo{\mathrm{Diffeo}}
\def\Symp{\mathrm{Symp}}
\def\Sympeo{\mathrm{Sympeo}}
\def\Id{\mathrm{Id}}
\newcommand{\norm}[1]{||#1||}
\def\Ham{\mathrm{Ham}}
\def\lagham#1{\mathcal{L}^\mathrm{\Ham}({#1})}
\def\Hamtilde{\widetilde{\mathrm{\Ham}}}
\def\cOlag#1{\mathrm{Sympeo}({#1})}
\def\Crit{\mathrm{Crit}}
\def\diag{\mathrm{diag}}
\def\Spec{\mathrm{Spec}}
\def\osc{\mathrm{osc}}
\def\id{\mathrm{id}}
\def\Cal{\mathrm{Cal}}
\def\Ker{\mathrm{Ker}}
\def\Im{\mathrm{Im}}
\def\Int{\mathrm{Int}}
\def\PD{\mathrm{PD}}
\def\Spec{\mathrm{Spec}}
\def\Hof{\mathrm{Hof}}

\title[]{Homogeneous quasimorphisms, $C^0$-topology and Lagrangian intersection}
\author{Yusuke Kawamoto}

\address{Yusuke Kawamoto: D\'epartement de Math\'ematiques et Applications, \'Ecole Normale Sup\'erieure, 45 rue d'Ulm, F 75230 Paris cedex 05}
\email{yusukekawamoto81@gmail.com, yusuke.kawamoto@ens.fr}

\maketitle

\begin{abstract}
We construct an example of a non-trivial homogeneous quasimorphism on the group of Hamiltonian diffeomorphisms of the two and four dimensional quadric hypersurfaces which is continuous with respect to both the $C^0$-metric and the Hofer metric. This answers a variant of a question of Entov--Polterovich--Py which is one of the open problems listed in the monograph of McDuff--Salamon. Throughout the proof, we make extensive use of the idea of working with different coefficient fields in quantum cohomology rings. As a by-product of the arguments in the paper, we answer a question of Polterovich--Wu regarding {\qmor}s on the group of Hamiltonian {\diffeo}s of the complex projective plane and prove some intersection results about Lagrangians in the four dimensional quadric hypersurface.
 \end{abstract}

\tableofcontents

\section{Introduction} 

A (real-valued) homogeneous quasimorphism on a group $G$ is a map
$$\mu:G\to \mathbb{R}$$
which satisfies
\begin{subequations}
\begin{gather}
    \exists C>0 \  \text{s.t.}\  \forall f,g \in G,\ |\mu(f\cdot g)-\mu(f)-\mu(g)|\leq C, \\
\forall k\in \mathbb{Z},\forall f\in G, \ \mu(f^k)=k\cdot \mu(f).
\end{gather}
\end{subequations}

The study of homogeneous quasimorphisms is a very rich topic with numerous connections to other mathematical domains.  For example, homogeneous quasimorphisms naturally appear in the theory of bounded cohomology, they play a crucial role in the study of the commutator length and they also have many applications in the study of algebraic and topological properties (in case $G$ is a topological group) of $G$. 

In the context of symplectic topology, the study of algebraic and topological properties of the group of symplectomorphisms and Hamiltonian diffeomorphisms has been an important subject. For a closed symplectic manifold $(M,\omega)$, denote the group of Hamiltonian diffeomorphisms by $\Ham(M,\omega)$ and its universal cover by $\widetilde{\Ham}(M,\omega)$. One of the first groundbreaking results in this direction is due to Banyaga \cite{[Ban78]} which states that $\Ham(M,\omega)$ is a simple group and $\widetilde{\Ham}(M,\omega)$ is a perfect group. This implies that there exist no non-trivial homomorphisms on $\Ham(M,\omega)$ and $\widetilde{\Ham}(M,\omega)$. However, it was discovered that non-trivial (real-valued) homogeneous quasimorphisms on $\Ham(M,\omega)$ and $\widetilde{\Ham}(M,\omega)$ do exist for some symplectic manifolds. Various constructions have been studied extensively as well as their applications to Hamiltonian dynamics. Just to mention a few, there are constructions by Barge--Ghys \cite{[BG92]}, Borman \cite{[Bor12]}, Entov \cite{[Ent04]}, Entov--Polterovich \cite{[EP03]}, Gambaudo--Ghys \cite{[GG04]}, Givental \cite{[Giv90]}, McDuff \cite{[McD10]}, Ostrover \cite{[Ost06]}, Py \cite{[Py06]} and Shelukhin \cite{[Sh14]}. Contact counterparts are also considered by Givental \cite{[Giv90]}, Borman--Zapolsky \cite{[BorZap15]} and Granja--Karshon--Pabiniak--Sandon \cite{[GKPS20]}. In particular, Entov--Polterovich \cite{[EP03]} introduced a Floer theoretic method to construct homogeneous quasimorphisms on $\widetilde{\Ham}(M,\omega)$
$$\zeta_e:  \widetilde{\Ham}(M,\omega) \to \mathbb{R}$$
where $(M,\omega)$ is a closed monotone symplectic manifold which satisfies some property. Recall that a closed symplectic manifold $(M,\omega)$ is called monotone if there exists a constant $\kappa>0$, which is referred to as the monotonicity constant, such that 
$$\omega|_{\pi_2(M)}=\kappa \cdot c_1|_{\pi_2(M)}$$
where $c_1=c_1(TM)$ denotes the first Chern class. In this paper, we only consider monotone symplectic manifolds unless mentioned otherwise. The precise construction of $\zeta_e$ is explained in Section \ref{Quasimorphism via spectral invariants}. Moreover, (a certain normalization of) $\zeta_e$ satisfies the so-called Calabi property which means, roughly speaking, that ``locally'' it coincides with the Calabi homomorphism: we refer to \cite{[EP03]} for the precise definition and its proof. In some cases, it is known that this  homogeneous quasimorphism descends to $\Ham(M,\omega)$. For an excellent survey of the theory of quasimorphisms in the symplectic context and their relations to other topics, we refer to \cite{[Ent14]}.

\section{Main results}\label{main results}

\subsection{Homogeneous quasimorphisms}
The following question concerning the continuity of quasimorphisms was posed by Entov--Polterovich--Py in \cite{[EPP12]}. This question appears also in the list of open problems in the monograph of McDuff--Salamon.

\begin{question}$($\cite{[EPP12]}, \cite[Chapter 14, Problem 23]{[MS98]}$)$
\begin{enumerate}
\item Does there exist a nonzero homogeneous quasimorphism 
$$\mu:\Ham(S^2) \to  \mathbb{R}$$
that is continuous with respect to the $C^0$-topology on $\Ham(S^2)$?
\item If yes, can it be made Lipschitz with respect to the Hofer metric?
\end{enumerate}
\end{question}

Recall that the $C^0$-topology on $\Ham(M,\omega)$ is induced by the $C^0$-metric of Hamiltonian diffeomorphisms $\phi,\psi \in \Ham(M,\omega)$, which is defined by
$$d_{C^0}(\phi,\psi):=\max_{x\in M}d(\phi(x),\psi(x)),$$
where $d$ denotes the distance on $M$ induced by a fixed Riemannian metric on $M$. See Section \ref{topology} for further remarks on $C^0$-topology as well as the Hofer metric.

We provide some background and motivation concerning this question of Entov--Polterovich--Py.

\begin{enumerate}
    \item \textit{Hofer metric vs. $C^0$-metric}: The relation between $C^0$-topology and the Hofer metric is very subtle. For example, $C^0$-topology is not continuous with respect to the Hofer metric. Conversely, Entov--Polterovich--Py point out that on $\Ham(D^{2n}(1))$, the group of compactly supported Hamiltonian diffeomorphisms of the closed unit ball $D^{2n}(1)$ in $\mathbb{R}^{2n}$, the Hofer metric is not $C^0$-continuous. For some striking results that demonstrate rigidity and flexibility of symplectic objects with respect to $C^0$-topology, see \cite{[BHS18]}, \cite{[BO16]} and \cite{[HLS15]}. 

In fact, for closed surfaces of positive genus $\Sigma$, there are examples of homogeneous quasimorphisms defined on $\Ham(\Sigma)$ which are $C^0$-continuous but not Hofer Lipschitz continuous: for their construction, see Gambaudo--Ghys \cite{[GG97]}, \cite{[GG04]} and for their discontinuity with respect to the Hofer metric, see \cite{[Kha19]}. On the other hand, the aforementioned Entov--Polterovich type homogeneous quasimorphisms are Hofer Lipschitz continuous but are not $C^0$-continuous: in fact, it is known that homogeneous quasimorphisms which have the Calabi property are not $C^0$-continuous: for a proof, see \cite{[EPP12]}.

\item \textit{Homogeneous quasimorphisms on the group of Hamiltonian homeomorphisms}: Given a symplectic manifold $(M,\omega)$, consider the $C^0$-closure of $\Ham(M,\omega)$ inside the group of homeomorphims of $M$. We denote it by $\overline{\Ham}(M,\omega)$ and call its elements Hamiltonian homeomorphisms. Hamiltonian homeomorphisms are central objects in $C^0$-symplectic topology. A $C^0$-continuous homogeneous quasimorphism defined on $\overline{\Ham}(M,\omega)$ will be useful to obtain information about the algebraic and topological properties of $\overline{\Ham}(M,\omega)$. In particular, when $(M,\omega)$ is either a $2$-sphere $S^2$ or a $2$-disk $D^2$, $\overline{\Ham}(M,\omega)$ is the identity component of the group of area-preserving homeomorphisms. A (non-trivial) homogeneous quasimorphism on $\overline{\Ham}(M,\omega)$ can be naturally obtained as an extension of a $C^0$-continuous (non-trivial) homogeneous quasimorphism on $\Ham(M,\omega)$ (see \cite[Proposition 1.4]{[EPP12]}). Therefore, the existence of a non-trivial $C^0$-continuous homogeneous quasimorphism on $\Ham(S^2)$ and $\Ham(D^2)$ has a strong relation to a question concerning the simplicity of groups $\overline{\Ham}(S^2)$ and $\overline{\Ham}(D^2)$ where the standard area-forms are considered as symplectic forms. The latter was known under the name of the simplicity conjecture (\cite[Chapter 14, Problem 42]{[MS98]}) and has caught the attention of many mathematicians over the years. It has been recently settled by Cristofaro-Gardiner--Humili\`ere--Seyfaddini \cite{[CGHS20]}.

\item \textit{Uniqueness of homogeneous quasimorphisms on $\Ham(S^2)$}: Another motivation is the uniqueness of homogeneous quasimorphism on $\Ham(S^2)$. For example, an affirmative answer to the first question will imply the non-uniqueness of such maps, since Entov--Polterovich type homogeneous quasimorphisms are not $C^0$-continuous.
\end{enumerate} 
For more background on this question, see \cite{[EPP12]}.

In this paper, we consider a generalized version of the question of Entov--Polterovich--Py: 

\begin{question} Does there exist a closed symplectic manifold $(M,\omega)$ which admits a non-trivial homogeneous quasimorphism on $\Ham(M,\omega)$ which is $C^0$-continuous? If yes, can it be Hofer Lipschitz continuous?
\end{question}

Entov--Polterovich--Py proved that the vector space consisting of non-trivial homogeneous quasimorphisms on $\Ham(D^{2n}(1))$ that are both $C^0$ and Hofer Lipschitz continuous is infinite dimensional \cite[Proposition 1.9]{[EPP12]}. However, no example of a closed symplectic manifold $(M,\omega)$ which admits a homogeneous quasimorphism on $\Ham(M,\omega)$ that is both Hofer continuous and $C^0$-continuous is known by the time of writing. In fact, for closed symplectic manifolds, according to \cite{[Ent14]}, constructions of Givental, Entov--Polterovich and Borman are so far the only known examples of homogeneous quasimorphisms (on $\widetilde{\Ham}(M,\omega)$) that are Hofer continuous. The Hofer continuity of Givental's homogeneous quasimorphisms was proven by Borman--Zapolsky \cite{[BorZap15]}. These examples all possess the Calabi property which implies that, in the case they descend to $\Ham(M,\omega)$, they are not $C^0$-continuous. The Calabi property of Givental's homogeneous quasimorphisms was proven by Ben Simon \cite{[BS07]}.

Our main result provides such examples for the monotone $n$-quadric $(Q^n,\omega)$ for $n=2,4$. Throughout the paper, we consider the standard monotone symplectic form $\omega$ of $Q^n$ with the normalization $\int_{Q^n} \omega^n=2$ so that the monotonicity constant $\kappa$ is $1/N_{Q^n}=1/n$. Note that $(Q^2,\omega)$ is symplectomorphic to the monotone product $(S^2\times S^2, \sigma \oplus \sigma)$ where $\sigma$ is the area-form of $S^2$ with $\int_{S^2} \sigma =1$ and $(Q^4,\omega)$ is symplectomorphic to $Gr_{\mathbb{C}}(2,4)$ equipped with the standard monotone symplectic form with a certain normalization. 

Precisely, we prove the following.

\begin{theo}\label{main theo}
There exist non-trivial homogeneous quasimorphisms
$$\mu:\Ham(Q^n) \to \mathbb{R}$$
where $n=2,4,$ that are $C^0$-continuous i.e. 
$$\mu:(\Ham(Q^n),d_{C^0}) \to \mathbb{R}$$
is continuous, and Lipschitz continuous with respect to the Hofer metric.
\end{theo}

\begin{remark}\label{novelty}
\begin{enumerate}
    \item Although it is not explicitly stated, the existence of a homogeneous quasimorphism on $\widetilde{\Ham}(Q^n), \forall n \in \mathbb{N}$ was essentially known since \cite{[EP03]}. The descent of Entov--Polterovich type homogeneous quasimorphisms to $\Ham(Q^2)$ and $\Ham(Q^4)$ was proven in \cite{[EP03]} and \cite{[Br11]}, respectively. The homogeneous quasimorphisms in Theorem \ref{main theo} are different from the Entov--Polterovich type homogeneous quasimorphisms as they are defined as differences of two Entov--Polterovich type homogeneous quasimorphisms. 
    
    \item In the case of $n=2$, if we compose $\mu: \Ham( S^2\times S^2) \to \mathbb{R}$ with
$$\Ham( S^2 ) \to \Ham( S^2\times S^2)$$ 
$$\phi \mapsto \phi \times \phi,$$
we obtain a $C^0$-continuous and Hofer Lipschitz continuous homogeneous quasimorphism on $\Ham( S^2 )$ but this turns out to be trivial and thus does not answer the question of Entov--Polterovich--Py. See Remark \ref{not good} for further explanation.
    
    \item In Section \ref{Proof of main theo general}, we will discuss a generalization of Theorem \ref{main theo}.
\end{enumerate}

\end{remark}

\subsection{Question of Polterovich--Wu}\label{A question of Polterovich--Wu}

One of the key ideas in proving Theorem \ref{main theo} and \ref{main theo general} is to work with quantum cohomology rings with different coefficient fields, namely the field of Laurent series and the universal Novikov field. The advantage of this idea in our context is explained in Section \ref{comparison lemma}. As another application of this idea, we answer a question of Polterovich--Wu which was posed in \cite[Remark 5.2]{[Wu15]}.

We briefly review the question. Details of the question are postponed to Section \ref{proof of Q of PW}. In \cite{[Wu15]}, Wu found three homogeneous quasimorphisms $\{ \zeta_{j} \}_{j=1,2,3}$ on $\widetilde{\Ham}(\mathbb{C}P^2)$ via the Entov--Polterovich construction for the quantum cohomology ring with the universal Novikov field. Polterovich posed the following question.

\begin{question}$($\cite[Remark 5.2]{[Wu15]}, see also Question \ref{precise Q of PW}$)$

Is it possible to distinguish the three homogeneous quasimorphisms $\{ \zeta_{j} \}_{j=1,2,3}$?
\end{question}

We answer this in the negative.

\begin{theo}\label{Q of PW}
The three homogeneous quasimorphisms $\{ \zeta_{j} \}_{j=1,2,3}$ coincide i.e.
$$\zeta_{1}=\zeta_{2}=\zeta_{3}.$$
\end{theo}

\subsection{Application}

The relation between the Hofer-topology and the $C^0$-topology on the group of Hamiltonian diffeomorphisms on closed symplectic manifolds still remains a mystery. In \cite{[LeR10]}, Le Roux posed the following question.

\begin{question}$($\cite{[LeR10]}$)$

Let $(M,\omega)$ be any closed symplectic manifold. For any $R>0$, does
$$\Ham_{\geq R}:=\{ \phi \in \Ham(M,\omega): d_{\Hof}( \id,\phi)\geq R\}$$
have a non-empty $C^0$-interior?
\end{question}

We answer to this question affirmatively for the quadric hypersurface $Q^n \times M$ where $n=2,4$.

\begin{theo}\label{leroux conj}

For any $R>0$,
$$\Ham_{\geq R}:=\{\phi \in \Ham(Q^n ): d_{\Hof}( \id,\phi)\geq R\}$$
has a non-empty $C^0$-interior where $n=2,4$.
\end{theo}

Theorem \ref{leroux conj} seems to be the first case where the question of Le Roux was verified for closed simply connected manifolds. In fact, according to \cite[Section 1.4]{[EPP12]}, ``for closed simply connected manifolds (and already for the case of the 2-sphere) the question is wide open''.

\begin{remark}
Our proof applies to any closed monotone symplectic manifold for which the spectral norm can be arbitrarily large: see Theorem \ref{monotone leroux}. See also Theorem \ref{leroux conj gene} for a slightly generalized statement.
\end{remark}

On a different note, Theorem \ref{main theo} has an application to the Rokhlin property of the group of Hamiltonian homeomorphisms. In fact, it implies that the group of Hamiltonian homeomorphisms of the two and four complex dimensional quadric hypersurfaces are not Rokhlin. We refer the readers to \cite[Section 1.1.1]{[Sh18]} on this topic.

\subsection{Strategy of the proof and structure of the paper}
The strategy of the proof of Theorem \ref{main theo}, which divides into two parts, is as follows:

We first prove that a {\qmor} on $\widetilde{\Ham}(M,\omega)$ which is obtained as the difference of any two Entov--Polterovich type {\qmor}s descends to $\Ham(M,\omega)$ and is bounded by the spectral norm $\gamma$. Next we show that it is $C^0$-continuous by using a result on the $C^0$-control of the spectral norm obtained by the author in \cite{[Kaw21]} (Theorem \ref{my lemma}). This is the first part of the proof. Note that this part applies to any monotone symplectic manifold.

In the second part of the proof, we will see that in order to prove that the resulting homogeneous quasimorphism is non-trivial, it suffices to find two disjoint Lagrangian submanifolds with non-vanishing Floer cohomology. We use examples found by Fukaya--Oh--Ohta--Ono and Eliashberg--Polterovich for the case of $Q^2$ and by Nishinou--Nohara--Ueda and Nohara--Ueda for the case of $Q^4$ where the Floer cohomology of Lagrangian fibers of a Gelfand--Cetlin system was studied via superpotential techniques. 

The crucial idea of the proof is to work with different quantum cohomology rings in Part 1 and 2. The differences of the two quantum cohomology rings as well as their advantages are explained in Section \ref{comparison lemma}. In Section \ref{proof of Q of PW}, we answer a question of Polterovich--Wu also by applying this idea. In Section \ref{lag int}, we discuss some consequences of the argument to  Lagrangian intersections.

\subsection{Acknowledgements}
I thank my supervisors Sobhan Seyfaddini and Claude Viterbo for regular discussions and for their precious comments on the earlier version of the paper. I am indebted to Kaoru Ono for a very stimulating discussion on the earlier version of this project, for explaining his works in collaboration with Fukaya, Oh and Ohta to me and for giving me an opportunity to present this work at a seminar at Kyoto University. Yoosik Kim has kindly answered several questions on his work for which I am grateful. I appreciate Egor Shelukhin for his comments on this work, especially for communicating the relation between Lemma \ref{rho compare} and a result of Usher--Zhang. I also benefited from conversations I had with Georgios Dimitroglou-Rizell, Jack Smith and Frol Zapolsky at the conference CAST 2020. I would like to express my gratitude to them as well. The referee has greatly contributed to improving the exposition to whom I owe a lot. I thank a lot for her/his careful reading and for the comments.

\section{Preliminaries}\label{preliminaries}
Let $(M,\omega)$ be a closed monotone symplectic manifold i.e.
$$\omega|_{\pi_2(M)}=\kappa \cdot c_1|_{\pi_2(M)}$$
for some monotonicity constant $\kappa>0$ where $c_1=c_1(TM)$ denotes the first Chern class. In this paper, we only consider monotone symplectic manifolds unless mentioned otherwise. The positive generators of $\langle \omega, \pi_2(M) \rangle$ and $\langle c_1, \pi_2(M) \rangle \subset \mathbb{Z}$ are respectively called the rationality constant and the minimal Chern number and will be respectively denoted by $\lambda_0$ and $N_M$. 

A Hamiltonian $H$ on $M$ is a smooth time dependent function $H:\mathbb{R}/\mathbb{Z} \times M\to \mathbb{R}$. A Hamiltonian $H$ is called mean-normalized if the following holds:
$$\forall t \in \mathbb{R}/\mathbb{Z},\ \int_M H_t(x)\omega^n=0.$$
We define its Hamiltonian vector field $X_{H_t}$ by 
$$-dH_t=\omega(X_{H_t},\ \cdot \ ).$$
The Hamiltonian flow of $H$, denoted by $\phi_H ^t$, is by definition the flow of $X_{H_t}$. A Hamiltonian diffeomorphism of $H$ is a diffeomorphism which arises as the time-one map of a Hamiltonian flow and will be denoted by $\phi_H$. It is well-known that the set of Hamiltonian diffeomorphisms forms a group and will be denoted by $\Ham(M,\omega)$. We denote its universal cover by $\widetilde{\Ham}(M,\omega)$.

Denote the set of smooth contractible loops in $M$ by $\mathscr{L}_0M$ and consider its universal cover. Two elements in the universal cover, say $[z_1,w_1]$ and $[z_2,w_2]$, are equivalent if $z_1=z_2$ and their boundary sum $w_1\# \overline{w_2}$ i.e. the sphere obtained by gluing $w_1$ and $w_2$ along their common boundary with the orientation on $w_2$ reversed, satisfies 
 $$\omega (w_1\# \overline{w_2})=0,\ c_1(w_1\# \overline{w_2})=0.$$
 We denote by $\widetilde{\mathscr{L}_0M}$ the space of equivalence classes.
 
For a Hamiltonian $H$, define the action functional $\mathscr{A}_H:\widetilde{\mathscr{L}_0 M}\to \mathbb{R}$ by
$$\mathscr{A}_H([z,w]):=\int_0 ^1 H(t,z(t))dt-\int_{D^2} w^{\ast}\omega$$
where $w:D^2\to M$ is a capping of $z:\mathbb{R}/\mathbb{Z}\to M$. Critical points of this functional are precisely the capped 1-periodic Hamiltonian orbits of $H$ which will be denoted by $\widetilde{\mathscr{P}}(H)$. The set of critical values of $\mathscr{A}_H$ is called the action spectrum and is denoted by $\Spec(H)$:
$$\Spec(H):=\{\mathscr{A}_H(\widetilde{z}):\widetilde{z} \in \widetilde{\mathscr{P}}(H)\}.$$

\subsection{Hofer and $C^0$ topologies on $\Ham(M,\omega)$}\label{topology}
Studying the topology of the group of Hamiltonian diffeomorphisms $\Ham(M,\omega)$ is an important topic in symplectic topology. In this section we recall two topologies of $\Ham(M,\omega)$.

The Hofer metric (or distance) is defined by 
$$d_{\Hof}(\phi,\psi):=\inf \{\int_0 ^1 ( \sup_x H_t(x)-\inf_x H_t(x)  ) dt:\phi_H=\psi^{-1}\circ \phi \}$$
for $\phi,\psi \in \Ham(M,\omega)$. The Hofer-topology is the topology induced by the Hofer metric.

The $C^0$-metric (or distance) of Hamiltonian diffeomorphisms $\phi,\psi \in \Ham(M,\omega)$ is defined by
$$d_{C^0}(\phi,\psi):=\max_{x\in M}d(\phi(x),\psi(x))$$
where $d$ denotes the distance on $M$ induced by a fixed Riemannian metric on $M$. The $C^0$-topology is the topology induced by the $C^0$-metric. Note that the $C^0$-topology is independent of the choice of the Riemannian metric.

\subsection{Hamiltonian Floer homology}\label{HF}

In this section, we give a quick overview of Floer theory. A standard reference is \cite{[MS04]}. We work with the ground field $\mathbb{C}$ in this paper. We say that a Hamiltonian $H$ is non-degenerate if the diagonal $\Delta:=\{(x,x)\in M\times M\}$ and the graph of $\phi_H$, $$\Gamma_{\phi_H}:=\{(x,\phi_H(x))\in M\times M\},$$
intersect transversally. We define the Floer chain complex of a non-degenerate Hamiltonian $H$, denoted by $CF_\ast(H)$ as follows: 
$$CF_\ast(H):=\{ \sum_{ \widetilde{z} \in \widetilde{\mathscr{P}}(H)} a_{\widetilde{z} } \cdot  \widetilde{z}  : \forall \tau \in \mathbb{R},\ \#\{\widetilde{z} : a_{\widetilde{z} } \in \C \backslash \{0\}, \mathscr{A}_H(\widetilde{z}) \leq \tau\}<+\infty \}.$$
The Floer chain complex $CF_\ast(H)$ is $\mathbb{Z}$-graded by the so-called Conley-Zehnder index $\mu_{CZ}$. The differential map counts certain solutions of a perturbed Cauchy-Riemann equation for a chosen $\omega$-compatible almost complex structure $J$ on $TM$, which can be viewed as isolated negative gradient flow lines of $\mathscr{A}_H$. This defines a chain complex $(CF_\ast (H),\partial)$ called the Floer chain complex whose homology is called the Floer homology of $(H,J)$ and is denoted by $HF_\ast (H,J)$. Often it is abbreviated to $HF_\ast (H)$ as Floer homology does not depend on the choice of an almost complex structure. Note that our convention of the Conley-Zehnder index is as follows:

Let $f$ denote a $C^2$-small Morse function. For every critical point $x$ of $f$, we require that
$$\mu_{CZ}([x,w_x])=i(x)$$
where $i$ denotes the Morse index and $w_x$ is the trivial capping.

Recapping a capped orbit $\widetilde{z}=[z,w]$ by gluing $A\in \pi_2(M)$ changes the action and the Conley-Zehnder index as follows:
\begin{subequations}
\begin{gather}
    \mathscr{A}_H([z,w\# A])=\mathscr{A}_H([z,w])-\omega(A).\\
\mu_{CZ}([z,w\# A])=\mu_{CZ}([z,w])-2c_1(A).
\end{gather}
\end{subequations}

We extend the action functional $\mathscr{A}_H$ as follows:
$$\mathscr{A}_H :CF_\ast (H) \to \mathbb{R}$$
$$\mathscr{A}_H(\sum_{ \widetilde{z} \in \widetilde{\mathscr{P}}(H)} a_{\widetilde{z} } \cdot  \widetilde{z} ):= \max_{a_{\widetilde{z} } \neq 0}  \mathscr{A}_H (\widetilde{z} ).$$

We then define the $\mathbb{R}$-filtered Floer chain complex of $H$ by the filtration of $\mathscr{A}_H$:
$$CF_\ast ^{\tau}(H):=\{ z \in CF_\ast (H) : \mathscr{A}_H(z) < \tau \}=\{ \sum_{\widetilde{z} \in \widetilde{\mathscr{P}}(H)}  a_{\widetilde{z}}  \widetilde{z} :  \mathscr{A}_H (\widetilde{z}) < \tau \text{ if } a_{\widetilde{z}} \neq 0 \}.$$
As the Floer differential map decreases the action, $(CF_\ast ^{\tau}(H),\partial)$ defines a chain complex whose homology is called the filtered Floer homology of $H$ and is denoted by $HF_\ast ^{\tau}(H)$.

\subsection{Quantum (co)homology and semi-simplicity}\label{Quantum (co)homology}
Consider a monotone symplectic manifold $(M,\omega)$. Let the following denote the field of Laurent series of a formal variable $s$:
$$ \mathbb{C}[[ s^{-1},s]:=\{ \sum_{k\leq k_0 } a_k s^k : k_0\in \mathbb{Z},a_k \in \mathbb{C} \}.$$

By identifying the variable $s$ with the generator of $\Gamma:=\pi_2(M)/ \sim$ where the equivalence relation is defined by $A,B \in \pi_2(M)$,
$$A\sim B \Longleftrightarrow \omega(A)=\omega(B)$$
satisfying
$$\omega(s)=\lambda_0,\ c_1(s)=N_M,$$ 
one can define the quantum homology ring $QH_\ast(M;\mathbb{C})$ as
$$QH_\ast(M;\mathbb{C}):=H_\ast(M; \mathbb{C}) \otimes_{\mathbb{C}} \mathbb{C}[[ s^{-1},s].$$
The quantum homology ring has the following valuation:
$$\nu_{QH_\ast} :QH_\ast(M;\mathbb{C}) \to \mathbb{R}$$
$$\nu_{QH_\ast}(\sum_{k\leq k_0 } a_k s^k):=\max \{k \cdot \omega(s)=k\cdot \lambda_0 :a_k \neq 0 \}.$$

Similarly, for a formal variable $t$, one can define the quantum cohomology ring $QH^\ast(M;\mathbb{C})$ as
$$QH^\ast(M;\mathbb{C}):=H^\ast(M; \mathbb{C}) \otimes_\mathbb{C} \mathbb{C}[ t^{-1},t]]$$
where 
$$ \mathbb{C}[ t^{-1},t]]:=\{ \sum_{k\geq k_0 } b_k t^k : k_0\in \mathbb{Z},b_k \in \mathbb{C} \}.$$
The quantum homology and quantum cohomology rings are isomorphic under the Poincar\'e duality map:
$$\PD:QH^\ast(M;\mathbb{C}) \xrightarrow{\sim} QH_{2n-\ast}(M;\mathbb{C})$$
$$a:=\sum_{k\geq k_0 } A_k t^k \mapsto \PD(a):=\sum_{k\geq k_0 } A_k^\# s^{-k}$$
where $\#$ denotes the usual Poincar\'e duality between singular homology and singular cohomology. Note that $t$ satisfies
$$\omega(t)=\lambda_0,\ c_1(t)=N_M.$$ 

The quantum cohomology ring has the following valuation:
$$\nu:=\nu_{QH^\ast} :QH^\ast(M;\mathbb{C}) \to \mathbb{R}$$
$$\nu(\sum_{k\geq k_0 } a_k t^k):=\min\{k \cdot \omega(t)=k\lambda_0 :a_k \neq 0 \}.$$

The ring structure of $QH_\ast(M;\mathbb{C})$ (and of $QH^\ast(M;\mathbb{C})$) is given by the quantum product which is denoted by $\ast$. It is defined by a certain count of pseudo-homolorphic spheres. More precisely, in the case of $QH_\ast(M;\mathbb{C})$,
$$\forall a,b,c \in H_\ast (M),\ (a\ast b)\circ c:=\sum_{k\in \mathbb{Z}} GW_{3,s^k}(a,b,c)\otimes s^k$$
where $\circ $ denotes the usual intersection index in homology and $GW_{3,s^k}(a,b,c)$ denotes the 3-pointed Gromov-Witten invariant for $a,b,c \in H_\ast (M)$ in the class $A\in  \pi_2(M)$ where $[A]=s^k \in \Gamma$ i.e. the count of pseudo-holomorphic spheres in the homotopy class $A$ passing through cycles representing $a,b,c \in H_\ast (M)$. See \cite{[MS98]} for details.

It is known that the Floer homology defined in Section \ref{HF} is canonically isomorphic to the quantum homology ring via the PSS-map:
$$PSS_H: QH_\ast(M;\mathbb{C}) \xrightarrow{\sim} HF_\ast (H).$$
Note that the PSS-map preserves the ring structure where the ring structure on RHS is given by the pair-of-pants product. See \cite{[MS04]} for details.

The quantum cohomology ring $QH^\ast(M;\mathbb{C})$ is called semi-simple if it splits into a finite direct sum of fields i.e. 
$$QH^\ast(M;\mathbb{C})=Q_1 \oplus Q_2 \oplus \cdots \oplus Q_l$$
for some $l\in \mathbb{N}$ where each $Q_j$ is a field. The identity $1\in QH^\ast(M;\mathbb{C})$ can then be decomposed into a sum of units $e_j \in Q_j$:
$$1=e_1 + e_2+\cdots +e_l.$$

\begin{remark}
The notion of semi-simplicity depends on the algebraic set-up of the quantum (co)homology. The notion explained above is the same as the one in \cite{[EP03]} which is not suitable to non-monotone settings as the Novikov ring is no longer a field. A more general notion of semi-simplicity was introduced in \cite{[Ost06]}, \cite{[EP08]}. \cite[Theorem 5.1]{[EP08]} states that in the monotone case, this generalized notion of semi-simplicity coincides with the one of \cite{[EP03]}.
\end{remark}

Examples of monotone symplectic manifolds whose quantum cohomology rings are semi-simple include $\mathbb{C}P^n$, 1, 2 and 3 point monotone blow-ups of $\mathbb{C}P^2$, complex Grassmanians $Gr_{\mathbb{C}}(2,n)$ and their products: see \cite{[EP03]}, \cite{[EP08]}.

Later, we will consider quantum cohomology with a different coefficient field, namely the universal Novikov field $\Lambda$ defined by
$$\Lambda:=\{\sum_{j=1} ^{\infty} a_j T^{\lambda_j} :a_j \in \mathbb{C}, \lambda _j  \in \mathbb{R},\lim_{j\to +\infty} \lambda_j =+\infty \}.$$
Fukaya--Oh--Ohta--Ono \cite{[FOOO09]}, \cite{[FOOO19]} study Floer theory with coefficients in $\Lambda$ rather than in the field of Laurent series and considers the following quantum cohomology:
$$QH^\ast(M;\Lambda):=H^\ast(M;\mathbb{C})\otimes_{\mathbb{C}} \Lambda.$$
It has the following valuation:
$$\nu:QH^\ast(M;\Lambda) \to \mathbb{R}$$
$$\nu(\sum_{j=1} ^{\infty} a_j T^{\lambda_j}):=\min\{\lambda_j: a_j \neq 0\}.$$
By considering
$$t \mapsto T^{+\lambda_0},$$
one can embed $QH^\ast(M;\mathbb{C})$ into $QH^\ast(M;\Lambda)$:
$$QH^\ast(M;\mathbb{C}) \hookrightarrow QH^\ast(M;\Lambda).$$

\subsection{Quantum homology of quadrics}\label{QH of quadrics}
In this section, we review some information about the quantum homology ring structure of quadric hypersurfaces. For $n \geq 2$, the $n$-quadric $Q^n$ is defined as a hypersurface in $\mathbb{C}P^{n+1}$ as follows:
$$Q^n:=\{(z_0:z_1:\cdots:z_{n+1}) \in \mathbb{C}P^{n+1}:z_0 ^2+z_1 ^2+\cdots+z_{n+1} ^2=0\}.$$
Recall that the minimal Chern number $N_{Q^n}$ of the $n$-quadric is $n$. It is well-known that $Q^2$ and $Q^4$ are respectively symplectomorphic to $S^2\times S^2$ and $Gr_{\mathbb{C}}(2,4)$. The ring structure of (quantum) homology of $Q^n$ can be found in \cite[Section 6.3]{[BC09]}. We just recall that $QH_\ast(Q^n;\mathbb{C})$ satisfies
$$[pt] \ast [pt] = [Q^n] s^{-2}$$
where $[pt]$ and $[Q^n]$ denote respectively the point class and the fundamental class. The semi-simplicity of the quantum homology ring of $Q^n$ follows from a result of Beauville \cite{[Bea95]}. In fact, it is easy to see that $QH_\ast(Q^n;\mathbb{C})$ splits into a direct sum of two fields by using that the minimal Chern number is $N_{Q^n}=n$.

\begin{prop}\label{QH(Q^n)}
For $n\geq 2$, $QH_\ast(Q^n;\mathbb{C})$ splits into a direct sum of two fields $Q_{\pm}$:
$$QH_\ast(Q^n;\mathbb{C})=Q_+ \oplus Q_-.$$
\end{prop}

\subsection{Hamiltonian spectral invariants}\label{spec inv}
In this subsection, we review spectral invariants and their basic properties. For a non-degenerate Hamiltonian $H$, let 
$$i ^{\tau}: CF_\ast ^\tau(H) \hookrightarrow CF_\ast (H) $$ be the natural inclusion map and denote the map it induces on homology by 
$$i_{\ast} ^{\tau}: HF_{\ast} ^\tau (H) \to HF_{\ast} (H).$$

For a quantum cohomology class $a\in QH^\ast(M;\mathbb{C})\backslash \{0\},$ define its spectral invariant by
$$\rho(H,a):=\inf \{\tau \in \mathbb{R}:  PSS_H\circ \PD(a)\in \Im(i_\ast ^\tau)\}.$$

The concept of spectral invariants was introduced by Viterbo for $\mathbb{R}^{2n}$ \cite{[Vit92]} in terms of generating functions and was later adapted to the Floer theoretic setting by Schwarz for symplectically aspherical manifolds \cite{[Sch00]} and by Oh for general closed symplectic manifolds \cite{[Oh05]}.

Spectral invariants are invariant under homotopy rel. end points i.e. if $t \mapsto \phi_H ^t$ and $t \mapsto \phi_G ^t$ are homotopic paths in $\Ham(M,\omega)$ where $H$ and $G$ are both mean-normalized Hamiltonians, then $\rho(H,\cdot)=\rho(G,\cdot)$. Thus, we can see spectral invariants as follows:
$$\rho:\widetilde{\Ham}(M,\omega) \times QH^\ast (M)\to \mathbb{R},$$
$$\rho(\tilde{\phi},a):=\rho(H,a)$$
for any mean-normalized $H$ such that the Hamiltonian path $t\mapsto \phi_H ^t$ represents the homotopy class $\tilde{\phi}$.

We list further properties of spectral invariants.

\begin{prop}\label{prop spec inv}
Spectral invariants satisfy the following properties where $H,G$ are Hamiltonians:
\begin{enumerate}
\item For any $a\in QH^\ast(M;\mathbb{C})\backslash \{0\},$
$$\mathscr{E}^- (H-G)\leq \rho(H,a)-\rho(G,a)\leq \mathscr{E}^+(H-G)$$ where
\begin{subequations}
\begin{gather}
    \mathscr{E}^-(H):=\int_{0} ^1 \inf_{x\in M} {H_t(x)}dt, \\
  \mathscr{E}^+(H):=\int_{0} ^1 \sup_{x\in M}{H_t(x)}dt, \\
 \mathscr{E}(H):=\mathscr{E}^+(H)-\mathscr{E}^-(H)=\int_0 ^1 \{ \sup_{x\in M}{H_t(x)}-\inf_{x\in M} {H_t(x)} \}dt
\end{gather}
\end{subequations}

\item If $H$ is non-degenerate, then for any $a\in QH^\ast(M;\mathbb{C})\backslash \{0\},$
$$\rho(H,a) \in \Spec(H).$$
Moreover, if $a\in QH^{\deg(a)}(M;\mathbb{C}),$ then there exists $\tilde{z} \in CF_{2n-\deg(a)}(H)$ such that
$$\rho(H,a)=\mathscr{A}_H (\tilde{z}).$$

\item For any $a\in QH^\ast(M;\mathbb{C})\backslash \{0\},$
$$\rho(0,a)=\nu(\PD(a))$$
where $0$ is the zero-function.

\item For any $a,b\in QH^\ast(M;\mathbb{C})\backslash \{0\},$
$$\rho(H\#G,a\ast b)\leq \rho(H,a)+\rho(G,b)$$
where 
$$(H\#G)(t,x):=H(t,x)+G(t,(\phi_H ^t)^{-1}(x))$$
and satisfies $\phi_{H \# G} ^t=\phi_H ^t \phi_G ^t$. 

\end{enumerate}
\end{prop}

\begin{remark}
A priori spectral invariants $\rho(H,\ \cdot \ )$ can be defined only if $H$ is non-degenerate as they are defined via Floer homology of $H$. However, by the continuity property i.e. Proposition \ref{prop spec inv} (1), one can define $\rho(H,\ \cdot \ )$ for any $H\in C^0(\mathbb{R}/\mathbb{Z}\times M,\mathbb{R})$ by considering an approximation of $H$ with non-degenerate Hamiltonians.
\end{remark}

The spectral pseudo-norm $\gamma$ for Hamiltonians is defined as follows:
$$\gamma: C^\infty(\mathbb{R}/\mathbb{Z}\times M,\mathbb{R}) \to \mathbb{R}_{\geq 0}$$
$$\gamma(H):=\rho(H,1)+\rho(\overline{H},1)$$
where $1\in QH^0(M;\mathbb{C})$ denotes the identity element of $QH^\ast (M;\mathbb{C})$. We can see the spectral pseudo-norm as a function on $\widetilde{\Ham}(M,\omega)$ as well:
$$\gamma: \widetilde{\Ham}(M,\omega) \to \mathbb{R}_{\geq 0}$$
$$\gamma(\widetilde{\phi}) :=\rho(\widetilde{\phi},1)+\rho(\widetilde{\phi}^{-1},1).$$

The spectral norm for Hamiltonian diffeomorphisms is defined by using the spectral pseudo-norm for Hamiltonians as follows:
$$\gamma: \Ham(M,\omega) \to \mathbb{R}_{\geq 0}$$
$$\gamma(\phi):=\inf_{\phi_H=\phi}\gamma(H).$$

Spectral invariants for Floer homology and quantum cohomology with $\Lambda$-coefficients were defined in a similar fashion in \cite{[FOOO19]} and they were proven to satisfy analogous properties to those listed in Proposition \ref{prop spec inv}. We refer to \cite{[FOOO19]} for details.

\subsection{Lagrangian Floer cohomology with bounding cochain}\label{deformed HF}
In this section, we sketch the construction of Lagrangian Floer cohomology deformed by a bounding cochain due to Fukaya--Oh--Ohta--Ono \cite{[FOOO09]}. In this paper, we mainly consider monotone Lagrangian submanifolds but it is worth mentioning that the theory of Fukaya--Oh--Ohta--Ono sketched in this section applies to any closed oriented Lagrangian submanifold which is relatively spin. We refer to \cite{[FOOO09]}, especially Chapter 3.1 for a detailed description of the material.

Let $L$ be a closed oriented Lagrangian submanifold with a fixed relatively spin structure. Recall that an oriented Lagrangian submanifold is relatively spin if its second Stiefel-Whitney class $w_2(TL)$ is in the image of the restriction map $H^2(M;\mathbb{Z}/2\mathbb{Z})\to H^2(L;\mathbb{Z}/2\mathbb{Z})$ (\cite[Definition 3.1.1]{[FOOO09]}). For example, if a Lagrangian is spin, then it is relatively spin and in particular, oriented Lagrangians are always relatively spin if $\dim_{\mathbb{R}} M\leq 6$.

Define the universal Novikov ring
$$\Lambda_0:=\{\sum_{j=1} ^{\infty} a_j T^{\lambda_j} :a_j \in \mathbb{C}, \lambda _j \geq 0,\lim_{j\to +\infty} \lambda_j =+\infty \}.$$
The universal Novikov field is given by
$$\Lambda:=\{\sum_{j=1} ^{\infty} a_j T^{\lambda_j} :a_j \in \mathbb{C}, \lambda _j  \in \mathbb{R},\lim_{j\to +\infty} \lambda_j =+\infty \}.$$
Define also
$$\Lambda_+:=\{\sum_{j=1} ^{\infty} a_j T^{\lambda_j} :a_j \in \mathbb{C}, \lambda _j  >0,\lim_{j\to +\infty} \lambda_j =+\infty \}.$$

Lagrangian intersection Floer theory equips the $\Lambda_0$-valued cochain complex of $L$ with the structure of an $A_\infty$-algebra. By taking the canonical model, one obtains an $A_\infty$-structure $\{ \mathfrak{m}_k \}_{0\leq k \leq \infty}$ on $H^\ast (L;\Lambda_0)$: we refer to \cite[Section 5.4]{[FOOO09]} for details. An element
$b\in H^1(L;\Lambda_+)$ is called a weak bounding cochain (in the sequel, we will simply call it a bounding cochain) if it satisfies the weak Maurer-Cartan equation
\begin{equation}\label{MC eq}
    \sum_{k=0} ^{\infty} \mathfrak{m}_k(b,b,\cdots,b)=0\ \mod \ \Lambda_0 \cdot \PD([L]).
\end{equation}

The set of (weak) bounding cochains will be denoted by $\widehat{\mathcal{M}}_{weak} (L)$. Note that $\widehat{\mathcal{M}}_{weak} (L)$ might be an empty set. We say that the Lagrangian $L$ is unobstructed if 
$$\widehat{\mathcal{M}}_{weak} (L) \neq \emptyset.$$
In the case $L$ is unobstructed, for any $b\in \widehat{\mathcal{M}}_{weak} (L)$, one can twist the Floer differential as
$$\mathfrak{m}_1 ^b (x):=\sum_{k,l \geq 0} \mathfrak{m}_{k+l+1}(b^{\otimes k} \otimes x \otimes b^{\otimes l}).$$
The Maurer-Cartan equation \ref{MC eq} implies 
$$\mathfrak{m}_1 ^b \circ \mathfrak{m}_1 ^b=0$$
and the resulting cohomology group

$$HF((L,b);\Lambda_0):=\frac{ \Ker (\mathfrak{m}_1 ^b:H^\ast (L;\Lambda_0) \to H^\ast (L;\Lambda_0))}{\Im(\mathfrak{m}_1 ^b:H^\ast (L;\Lambda_0)\to H^\ast (L;\Lambda_0))}$$
will be called the Floer cohomology deformed by a (weak) bounding cochain $b\in \widehat{\mathcal{M}}_{weak} (L)$.
We also define 
$$HF((L,b);\Lambda):=HF((L,b);\Lambda_0) \otimes_{\Lambda_0} \Lambda.$$

\subsection{Quasimorphisms via spectral invariants}\label{Quasimorphism via spectral invariants}
In this subsection, we recall the Floer theoretic construction of homogeneous quasimorphisms on $\widetilde{\Ham}(M,\omega)$ and the notion of (super)heaviness both due to Entov--Polterovich which are taken from \cite{[EP03]}, \cite{[EP09]}. However, unlike their version, we use quantum cohomology instead of quantum homology.

Assume $e\in QH^0(M;\mathbb{C})$ is an idempotent. Then we define the asymptotic spectral invariant
$$\zeta_e:C^\infty(\mathbb{R}/\mathbb{Z}\times M, \mathbb{R}) \to \mathbb{R}$$
$$\zeta_e(H):=\lim_{k\to +\infty} \frac{\rho(H^k,e)}{k}$$
where $\rho(\cdot,e)$ denotes the spectral invariant corresponding to $e\in QH^0 (M;\mathbb{C})$ and the $k$-times iterated Hamiltonian
$$H^k:=\underbrace{H \# H \# \cdots \# H}_{k\text{-times}}.$$
Its restriction to $C^\infty( M, \mathbb{R})$ i.e. $\zeta_e|_{C^\infty( M, \mathbb{R})}: C^\infty( M, \mathbb{R}) \to \mathbb{R}$ is often referred to the symplectic quasi-state \cite{[EP06]}. 

We can also see $\zeta_e$ as a function of homotopy classes of Hamiltonian paths:
$$\zeta_e: \widetilde{\Ham}(M,\omega)\to \mathbb{R}$$
$$\zeta_e(\tilde{\phi}):=\lim_{k\to +\infty} \frac{\rho(\tilde{\phi}^k,e)}{k}.$$
Recall that $\rho(\tilde{\phi},\ \cdot \ )=\rho(H,\ \cdot \ )$ where $H$ is the mean-normalized Hamiltonian such that the Hamiltonian path $t\mapsto \phi_H ^t$ represents the homotopy class $\tilde{\phi}$. It was first discovered by Entov--Polterovich that when some additional condition is satisfied, $\zeta_e: \widetilde{\Ham}(M,\omega)\to \mathbb{R}$ is a homogeneous quasimorphism. We will state their result as well as its variant due to Fukaya--Oh--Ohta--Ono.

We denote the even degree part of $QH^{\ast}(M;\mathbb{C})$ as follows:
$$QH^{even}(M;\mathbb{C}):=\bigoplus_{ k\in \mathbb{Z} } H^{2k}(M;\mathbb{C}) \otimes_{\mathbb{C}} \mathbb{C}[t^{-1},t]].$$

\begin{theo}\label{EP quasimorphism}$($\cite[Theorem 1.1]{[EP03]}, \cite{[FOOO19]}$)$
\begin{enumerate}
\item If $e\in QH^0(M;\mathbb{C})$ is an idempotent and $e\cdot QH^{even}(M;\mathbb{C})$ is a field, then 
$$\zeta_e: \widetilde{\Ham}(M,\omega)\to \mathbb{R}$$
is a homogeneous quasimorphism. 

\item If $e\in QH^\ast(M;\Lambda)$ is an idempotent and $e\cdot QH^\ast(M;\Lambda)$ is a field, then 
$$\zeta_e: \widetilde{\Ham}(M,\omega)\to \mathbb{R}$$
is a homogeneous quasimorphism.
\end{enumerate}
\end{theo}

\begin{remark}
All the examples that appear in this paper satisfy 
$$QH^{even}(M;\mathbb{C})=QH^{\ast}(M;\mathbb{C}).$$
\end{remark}

\begin{definition}
Let $(M,\omega)$ be any closed symplectic manifold and let $e\in QH^\ast(M;\mathbb{C})$ be an idempotent. A subset $S$ of $M$ is called $e$-heavy or $\zeta_e$-heavy (resp. $e$-superheavy or $\zeta_e$-superheavy) if it satisfies the following:
$$ \inf_{x\in S} H(x) \leq \zeta_e(H) $$
$$(\text{resp.}\ \zeta_e(H) \leq \sup_{x\in S} H(x))$$
for any $H\in C^\infty (M,\mathbb{R}).$
\end{definition}

\begin{remark}
In general, $e$-heaviness follows from $e$-superheaviness but not vice versa. In a special case where $\zeta_e: \widetilde{\Ham}(M,\omega)\to \mathbb{R}$ is a homogeneous quasimorphism, $e$-heaviness and $e$-superheaviness are equivalent. See \cite{[EP09]} for discussions in this topic.
\end{remark}

The following is a basic intersection property of (super)heavy sets from \cite{[EP09]}.

\begin{prop}\label{disjoint-heavy}
Let $(M,\omega)$ be any closed symplectic manifold and let $e\in QH^\ast(M;\mathbb{C})$ be an idempotent. Let $S_1$ and $S_2$ be two disjoint subsets of $M$. If $S_1$ is $e$-superheavy, then $S_2$ is not $e$-heavy.
\end{prop}  

\begin{proof}
If we assume that $S_2$ is $e$-heavy, then by the definitions, we have
$$\inf_{x\in S_2} H(x) \leq \zeta_e(H) \leq \sup_{x\in S_1} H(x)$$
for any $H\in C^{\infty}(M,\mathbb{R})$. As $S_1\cap S_2 = \emptyset$, one can take $H$ to be larger on $S_2$ than on $S_1$, which contradicts the inequality.
\end{proof}

\subsection{Closed-open map and heaviness}
In this section, we review some properties of the closed-open map defined by Fukaya--Oh--Ohta--Ono in \cite[Theorem 3.8.62]{[FOOO09]}. Note that they also consider the case where the absolute and the relative Floer cohomology groups are deformed with a bulk. However, as bulk deformations are not relevant to the arguments in this paper, we only state a version without them.

Denote the ring homomorphism called the closed-open map, which is a quantum analogue of the restriction map, by
$$\mathcal{CO}^0 _b: QH^\ast (M;\Lambda) \to HF^\ast ((L,b);\Lambda)$$
where $b$ is a bounding cochain. Note that the original notation used in \cite{[FOOO09]} for $\mathcal{CO}^0 _b$ is $i^\ast _{qm,b}$.

Fukaya--Oh--Ohta--Ono proved the following in \cite{[FOOO19]} to detect the heaviness of the Lagrangian $L$, which generalizes the result of Albers \cite{[Alb05]} and Entov--Polterovich \cite[Theorem 1.17]{[EP09]}.

\begin{theo}\label{co map}$($\cite[Theorem 1.6]{[FOOO19]}$)$

Assume
$$HF^\ast ((L,b);\Lambda)\neq 0$$
for a certain bounding cochain $b$. 
If 
$$\mathcal{CO}^0 _b(e)\neq 0$$
for an idempotent $e\in QH^\ast (M;\Lambda)$, then $L$ is $e$-heavy.
\end{theo}

\subsection{Flag manifolds and Gelfand--Cetlin systems}\label{GC system}

In this subsection, we provide a brief description of (partial) flag manifolds and Gelfand--Cetlin systems. Materials discussed in this section are only needed to precisely understand the statement of Theorem \ref{ueda} and will not be used in other parts of the paper. Thus, readers can skip this section in order to read the other parts.

Fix a sequence 
$$0 = n_0 < n_1 < \cdots < n_r < n_{r+1} = n$$
of integers, and set 
$$k_i := n_i - n_{i-1}$$
for $i = 1,2,\cdots, r +1$. The (partial) flag manifold $F = F(n_1,n_2,\cdots, n_r,n)$ is a complex manifold parametrizing nested subspaces
$$0 \subset V_1 \subset V_2 \subset \cdots \subset V_r \subset \mathbb{C}^n,\ \dim V_i =n_i.$$
The dimension of $F = F(n_1,n_2,\cdots, n_r,n)$ is given by
\begin{equation}\label{dim}
    \dim_{\mathbb{C}} F(n_1,n_2,\cdots, n_r,n)=\sum_{i=1} ^r (n_i - n_{i-1})(n-n_i ) = \sum_{i=1} ^r k_i  (n-n_i ) .
\end{equation}

Let $P = P(n_1,n_2,\cdots ,n_r,n)\subset  GL(n,\mathbb{C})$ be the isotropy subgroup of the standard flag
$$\mathbb{C}^{n_1}\times \{0\} \subset \mathbb{C}^{n_2}\times \{0\} \subset \cdots \subset \mathbb{C}^{n_r}\times \{0\} \subset \mathbb{C}^n.$$

Then, as 
$$U(n) \cap P(n_1,n_2,\cdots ,n_r,n)=U(k_1)\times U(k_2)\times \cdots U(k_{r+1}),$$
$F(n_1,n_2,\cdots, n_r,n)$ is written as follows:
$$F(n_1,n_2,\cdots, n_r,n)=GL(n,\mathbb{C})/  P(n_1,n_2,\cdots ,n_r,n)$$
$$=U(n)/(U(k_1)\times U(k_2)\times \cdots U(k_{r+1})).$$
 \begin{remark}
 Note that this description gives the following different expression of the dimension formula \ref{dim}:
 $$\dim_{\mathbb{C}} F(n_1,n_2,\cdots, n_r,n)= n^2 - \sum_{i=1} ^{r+1} k_i ^2 .$$
 
 \end{remark}

In this paper, we identify flag manifolds with (co)adjoint orbits. Using a $U(n)$-invariant inner product on the Lie algebra $ \mathfrak{u}(n)$ of $U(n)$, denoted by $\langle - , - \rangle$, we identify the dual $ \mathfrak{u}(n)^\ast$ of $ \mathfrak{u}(n)$ with the space $\sqrt{-1}\cdot \mathfrak{u}(n)$ of Hermitian matrices. We fix 
$$\lambda= \diag ( \lambda_1,\lambda_2,\cdots,\lambda_n) \in \sqrt{-1}\cdot \mathfrak{u}(n) $$
with

$$\underbrace{\lambda_1 = \cdots = \lambda_{n_1}}_{k_1} > \underbrace{\lambda_{n_1 +1} = \cdots= \lambda_{n_2} }_{k_2}> \cdots >\underbrace{ \lambda_{n_r +1} = \cdots = \lambda_n}_{k_{r+1}}.$$

Then $F$ is identified with the adjoint orbit $\mathcal{O}_\lambda$ of $ \lambda$ (i.e. a set of Hermitian matrices with fixed eigenvalues $\lambda_1,\lambda_2,\cdots , \lambda_n$) by
$$F = U(n)/(U(k_1)\times \cdots \times U(k_{r+1})) \xrightarrow[]{\sim} \mathcal{O}_\lambda$$
$$[g] \mapsto g  \lambda g^\ast.$$

$\mathcal{O}_ \lambda$ has a standard symplectic form $\omega_\lambda$ called the Kirillov--Kostant--Souriau form. Recall that tangent vectors of $\mathcal{O}_ \lambda$ at $x$ can
be written as 
$$ad_\xi (x) = [x, \xi]$$
for $\xi \in  \mathfrak{u}(n)$ where $[-,-]$ denotes the Lie bracket. Then the Kirillov--Kostant--Souriau form $\omega_\lambda$ is defined by

$$\omega_\lambda(ad_\xi (x),ad_\eta (x)):=\frac{1}{2\pi} \langle x, [\xi,\eta] \rangle.$$

The following choice of $\lambda$ gives us a monotone symplectic form $\omega_\lambda$ on $\mathcal{O}_\lambda$:

$$\lambda=(\underbrace{n-n_1,\cdots}_{k_1},\underbrace{ n-n_1-n_2,\cdots}_{k_2}, \cdots,\underbrace{\cdots, n-n_{r-1}-n_r}_{k_r},\underbrace{-n_r,\cdots}_{k_{r+1}})$$
$$+\underbrace{(m,\cdots,m)}_{n=k_1+\cdots+ k_{r+1}}$$
for any $m\in \mathbb{R}$. When $\lambda$ is of this form, we have
$$c_1(T\mathcal{O}_\lambda)=[\omega_\lambda].$$

For $x\in \mathcal{O}_ \lambda$ and $k=1,2,\cdots,n-1$ let $x^{(k)}$ denote the upper-left $k\times k$ submatrix of $x$. Since $x^{(k)}$ is also a Hermitian matrix, it has real eigenvalues 
$$\lambda_1 ^{(k)} \leq \lambda_2 ^{(k)}\leq \cdots \leq \lambda_k ^{(k)}.$$

Let $I=I(n_1\cdots,n_r, n)$ denote the set of pairs $(i, k)$ such that each $\lambda_i ^{(k)}$ is non-constant as a function of $x$. It follows that the number of such pairs coincides with $\dim_\mathbb{C} F$ i.e. $|I|=\dim_\mathbb{C} F$. The Gelfand--Cetlin system is defined by
$$\Phi: F\to \mathbb{R}^{\dim_{\mathbb{C}}F}$$
$$\Phi (x) :=\{\lambda_i ^{(k)} (x) \}_{(i,k)\in I}$$

\begin{theo}$($Guillemin--Sternberg, \cite{[GS83]}$)$

The map $\Phi$ defines a completely integrable system on $(F(n_1,n_2,\cdots,n_r,n),\omega)$.
The image $\Delta:=\Phi(F)$ is a convex polytope. A fiber of each interior point $u\in \Int(\Delta)$ is a Lagrangian torus:
$$\Phi^{-1}(u) \simeq T^n$$
for any $u\in \Int(\Delta)$.
\end{theo}

We call the convex polytope $\Delta:=\Phi(F)$, the Gelfand--Cetlin polytope. The major difference between Delzant polytopes of toric manifolds and Gelfand--Cetlin polytopes appears at fibers of points at the boundary of polytopes. While for a Delzant polytope, a fiber of a relative interior of a $k$-dimensional face is never Lagrangian, for a Gelfand--Cetlin polytope, a fiber of a relative interior point of a $k$-dimensional face can be a (non-torus) Lagrangian submanifold. Differences between the two types of polytopes are listed by Y. Cho--Y. Kim--Y-G. Oh in \cite{[CKO18]}.

\section{Proofs}

\subsection{Proof of Theorem \ref{main theo}--Part 1}\label{part1}
The goal of this subsection is to prove the following result and to see how it leads to Theorem \ref{main theo}.

\begin{theo}\label{general}
Let $(M,\omega)$ be a monotone symplectic manifold. Assume its quantum cohomology ring $QH^\ast(M;\mathbb{C})$ is semi-simple i.e. 
$$QH^\ast(M;\mathbb{C})=Q_1 \oplus Q_2 \oplus \cdots \oplus Q_l$$
for some $l\in \mathbb{N}$ where each $Q_j$ is a field. We decompose the identity $1 \in QH^\ast(M;\mathbb{C})$ into a sum of idempotents with respect to this split:
$$1=e_1+e_2+\cdots+e_l,\ e_j\in Q_j.$$
Then for any $i,j \in \{1,2,\cdots,l\},$
$$\mu:=\zeta_{e_i}-\zeta_{e_j}$$
defines a homogeneous quasimorphism on $\Ham(M,\omega)$ which is $C^0$-continuous i.e.
$$\mu :(\Ham(M,\omega),d_{C^0})\to \mathbb{R}$$ is continuous. Moreover, it is Hofer Lipschitz continuous.
\end{theo}

\begin{remark}
\begin{enumerate}
\item As we do not know if $\zeta_{e_i}\neq \zeta_{e_j}$, the resulting homogeneous quasimorphism
$$\mu : \Ham(M,\omega) \to \mathbb{R}$$
might be trivial i.e. $\mu \equiv 0$. Thus the point in proving Theorem \ref{main theo} is to prove $\zeta_{e_+}\neq \zeta_{e_-}$ for the two idempotents $e_\pm \in QH^\ast(Q^n;\mathbb{C}) \ (n=2,4)$. 

\item For examples of monotone symplectic manifolds whose quantum cohomology ring is semi-simple, see Section \ref{Quantum (co)homology}.

\item In the spirit of McDuff \cite{[McD10]}, instead of the semi-simplicity we can pose a weaker assumption that $QH^\ast(M;\mathbb{C})$ has two fields as a direct summand:
$$QH^\ast(M;\mathbb{C})=Q_1 \oplus Q_2\oplus A$$
where $Q_1,Q_2$ are fields and no condition is posed on $A$.

\end{enumerate}

\end{remark}

We first show the following estimate.

\begin{prop}\label{mu leq gamma}
For any $\tilde{\phi} \in \widetilde{\Ham}(M,\omega)$, we have
$$|\mu(\tilde{\phi})|\leq \gamma(\tilde{\phi}).$$

\end{prop}

\begin{proof}[Proof of Proposition \ref{mu leq gamma}]

By the triangle inequality,
\begin{subequations}
\begin{gather}
    \rho(\tilde{\phi}^k,e_1)\leq \rho(\tilde{\phi}^k,1)+\rho(\tilde{\id},e_1), \\
 -\rho(\tilde{\phi}^k,e_2) \leq \rho((\tilde{\phi}^{-1})^k,1)-\rho(\tilde{\id},e_2).
\end{gather}
\end{subequations}
By adding these inequalities, we obtain
$$\mu(\tilde{\phi}) = \lim_{k\to +\infty} \frac{\rho(\tilde{\phi}^k,e_1)-\rho(\tilde{\phi}^k,e_2)}{k} $$
$$\leq  \lim_{k\to +\infty} \frac{\rho(\tilde{\phi}^k,1)+\rho(\tilde{\id},e_1)+\rho((\tilde{\phi}^{-1})^k,1)-\rho(\tilde{\id},e_2)}{k} $$
$$=\lim_{k\to +\infty} \frac{\gamma(\tilde{\phi}^k)+\nu(e_1)-\nu(e_2)}{k} \leq \gamma(\tilde{\phi}).$$
As $\mu$ is homogeneous, we have
$$-\mu(\tilde{\phi})=\mu(\tilde{\phi}^{-1})$$
for any $\tilde{\phi}$ and thus
$$-\mu(\tilde{\phi})=\mu(\tilde{\phi}^{-1}) \leq \gamma(\tilde{\phi}^{-1})=\gamma(\tilde{\phi}).$$
Thus,
$$|\mu(\tilde{\phi})|\leq \gamma(\tilde{\phi}).$$
This completes the proof of Proposition \ref{mu leq gamma}.
\end{proof}

One can strengthen the statement as follows.
 
\begin{prop}\label{descend}
The function $$\mu:\widetilde{\Ham}(M,\omega) \to \mathbb{R}$$ descends to $\Ham(M,\omega)$ i.e. if $\tilde{\phi}$ and $\tilde{\psi}$ have the same endpoint, then $\mu(\tilde{\phi})=\mu(\tilde{\psi}).$ Thus, for any $\phi \in \Ham(M,\omega),$ we define
$$\mu(\phi):=\mu(\tilde{\phi})$$
where $\tilde{\phi} \in \widetilde{\Ham}(M,\omega)$ is any element having $\phi$ as the endpoint. We can thus define a map
$$\mu:\Ham(M,\omega) \to \mathbb{R}.$$
It satisfies
$$|\mu(\phi)|\leq \gamma(\phi)$$
for any $\phi \in \Ham(M,\omega).$
\end{prop}

\begin{proof}[Proof of Proposition \ref{descend}]

It suffices to show $\mu|_{\pi_1(\Ham(M,\omega))} \equiv 0$ where we see $\pi_1(\Ham(M,\omega)) \subset \widetilde{\Ham}(M,\omega)$. This is for the following reason.

Assume $\mu|_{\pi_1(\Ham(M,\omega))} \equiv 0.$ Let $\tilde{\phi},\tilde{\psi}$ be two homotopy classes of Hamiltonian paths having the same endpoint. For any $k\in \mathbb{N},$ $(\tilde{\phi}^{-1})^k \tilde{\psi}^k$ defines a homotopy class of a Hamiltonian loop i.e. an element in $\pi_1(\Ham(M,\omega))$. Since $\mu$ is a quasimorphism on $\widetilde{\Ham}(M,\omega)$, there exists a constant $C>0$ such that
$$|\mu((\tilde{\phi}^{-1})^k \tilde{\psi}^k)-\mu(\tilde{\psi}^k)-\mu((\tilde{\phi}^{-1})^k)|\leq C$$ for any $k\in \mathbb{N}$. From our assumption, the first term vanishes and 
$$\mu(\tilde{\psi}^k)=k\cdot \mu(\tilde{\psi}),$$
$$\mu((\tilde{\phi}^{-1})^k)=-k\cdot \mu(\tilde{\phi}).$$ 
Thus, we have
$$\forall k\in \mathbb{N},\ k\cdot |\mu(\tilde{\psi})-\mu(\tilde{\phi})|\leq C.$$
Therefore, we attain $\mu(\tilde{\phi})=\mu(\tilde{\psi}).$

Now, we prove $\mu|_{\pi_1(\Ham(M,\omega))} \equiv 0.$ We make use of the following theorem proved in \cite{[Kaw21]}. We restate it with a special emphasis on a particular case which will be used in our argument:

\begin{theo}\label{my lemma}$($\cite[Theorem 4(1)]{[Kaw21]}$)$

Let $(M,\omega)$ be a monotone symplectic manifold. For any $\varepsilon>0$, there exists $\delta>0$ such that if $d_{C^0}( \id,\phi_H)<\delta$, then 
$$\gamma(H)<\frac{\dim(M)}{N_M}\cdot \lambda_0 +\varepsilon$$
where $N_M$ denotes the minimal Chern number. In particular, for any $\psi \in \pi_1(\Ham(M,\omega)),$ we have
$$\gamma(\psi) \leq \frac{\dim(M)}{N_M}\cdot \lambda_0.$$
\end{theo}

Now we continue the proof of Proposition \ref{descend}. Let $\psi \in \pi_1(\Ham(M,\omega)).$ For any $k \in \mathbb{N}$, we have
$$k\cdot |\mu(\psi)|=|\mu(\psi^k)| \leq \gamma(\psi^k) \leq \frac{\dim(M)}{N_M}\cdot \lambda_0.$$
Thus,
$$|\mu(\psi)|\leq \lim_{k\to +\infty} \frac{\dim(M)}{N_M} \cdot \frac{\lambda_0}{k}=0.$$
This completes the proof of the first assertion. The second follows immediately from Proposition \ref{mu leq gamma}.
\end{proof}

\begin{remark}
The estimate of the spectral norm for Hamiltonian loops that appear in Theorem \ref{my lemma} can be deduced by using basic facts about the Seidel elements as well.
\end{remark}

We will use the following criterion due to Shtern to detect the $C^0$-continuity of homogeneous quasimorphisms, see \cite{[Sht01]} and \cite[Proposition 1.3]{[EPP12]}.

\begin{prop}\label{shtern}$($\cite{[Sht01]}, \cite[Proposition 1.3]{[EPP12]}$)$

Let $G$ be a topological group and $\mu:G\to \mathbb{R}$ a homogeneous quasimorphism. Then $\mu$ is continuous if and only if it is bounded on a neighborhood of the identity.
\end{prop}

We now complete the proof of Theorem \ref{general}.

\begin{proof}[Proof of Theorem \ref{general}]

By Propositions \ref{descend} and \ref{shtern}, the $C^0$-continuity of $\mu:\Ham(M,\omega) \to \mathbb{R}$ is reduced to the boundedness of the spectral norm $\gamma$ around a $C^0$-neighborhood of $\id$. Theorem \ref{my lemma} implies that the spectral norm is bounded around the identity of $\widetilde{\Ham}(M,\omega)$ (thus on $\Ham(M,\omega)$ as well) with respect to the $C^0$-topology when $(M,\omega)$ is monotone and therefore, $\mu$ is $C^0$-continuous. As $\zeta_{e_i}$ and $\zeta_{e_j}$ are both Hofer Lipschitz continuous, $\mu$ is also Hofer Lipschitz continuous. This completes the proof of Theorem \ref{general}.

\end{proof}

By Proposition \ref{QH(Q^n)}, $QH^\ast(Q^n;\mathbb{C})$ is semi-simple and splits into a direct sum of two fields
$$QH^\ast(Q^n;\mathbb{C})=Q_+ \oplus Q_-$$
and we decompose the identity element as follows:
$$1=e_+ + e_-.$$
By the Entov--Polterovich theory, we obtain homogeneous (Calabi) quasimorphisms
$$\zeta_{e_\pm}: \widetilde{\Ham}(Q^n) \to \mathbb{R}$$
$$\zeta_{e_\pm}(\widetilde{\phi}):=\lim_{k\to +\infty} \frac{\rho(\widetilde{\phi} ^k,e_\pm)}{k}.$$

In the second part of the proof (Section \ref{part2}), we will prove the following.

\begin{theo}\label{theo part2}
For $Q^n\ (n=2,4)$,  
$$\zeta_{e_+} \neq \zeta_{e_-}.$$
\end{theo}

Once we prove this, Theorems \ref{general} and \ref{theo part2} imply that
$$\mu:=\zeta_{e_+} - \zeta_{e_-}$$
defines a $non$-$trivial$ homogeneous quasimorphism on $\Ham(Q^n)\ (n=2,4)$ which is both $C^0$ and Hofer Lipschitz continuous and we complete the proof of Theorem \ref{main theo}.

\begin{remark}\label{not good}
As remarked in Remark \ref{novelty} (2), the composition of $\mu: \Ham( S^2\times S^2) \to \mathbb{R}$ and
$$\Ham( S^2 ) \to \Ham( S^2\times S^2)$$ 
$$\phi \mapsto \phi \times \phi$$
vanishes. This is because, by Proposition \ref{descend}, we have
$$|\mu( \phi \times \phi)| \leq \gamma( \phi \times \phi) = 2 \gamma(\phi)$$
for any $\phi \in \Ham( S^2 )$. Note that the first and the second $\gamma$ both denote the spectral norm but the former is for $\Ham( S^2\times S^2)$ and the latter is for $\Ham( S^2 )$. As remarked in Remark \ref{remark after finite gamma}, the spectral norm for $\Ham( S^2 )$ is bounded and thus the homogeneity of $\mu$ implies $\mu( \phi \times \phi)=0$ for any $\phi \in \Ham( S^2 )$.
\end{remark}

\subsection{Comparing different quantum cohomology rings}\label{comparison lemma}

In the first part of the proof of Theorem \ref{main theo}, we have used the quantum cohomology ring denoted by $QH^\ast(M;\mathbb{C})$ but in the second part of the proof, we work with a different quantum cohomology ring, namely the quantum cohomology ring with the universal Novikov field which is denoted by $QH^\ast(M;\Lambda)$. In this section, we explain the different advantages of working with $QH^\ast(M;\mathbb{C})$ and $QH^\ast(M;\Lambda)$. Working with these two different quantum cohomology rings plays a crucial role not only in the proof of Theorem \ref{main theo} but also in the proof of Theorem \ref{Q of PW}. We also compare spectral invariants of a quantum cohomology class in $QH^\ast(M;\mathbb{C})$ and its embedded quantum cohomology class in $QH^\ast(M;\Lambda)$. Note that results in this subsection concern not only the $n$-quadric but any monotone symplectic manifold.

Let $(M,\omega)$ be a monotone symplectic manifold. Recall from Section \ref{Quantum (co)homology} that $QH^\ast(M;\mathbb{C})$ was defined by
$$QH^\ast(M;\mathbb{C}):=H^\ast(M;\mathbb{C}) \otimes_{\C} \C [ t^{-1},t|]$$
where the variable $t$ represents an element in $\pi_2(M)$ that satisfies
$$\omega(t) = \lambda_0,\ c_1(t) = N_M.$$
On the other hand $QH^\ast(M;\Lambda)$ is defined by
$$QH^\ast(M;\Lambda):= H^\ast(M;\C) \otimes_{\C} \Lambda$$
and one can embed $QH^\ast(M;\mathbb{C})$ to $QH^\ast(M;\Lambda)$ by
$$\iota:QH^\ast(M;\mathbb{C}) \hookrightarrow QH^\ast(M;\Lambda)$$
$$t \mapsto T^{+\lambda_0}$$
and $\iota$ is a ring homomorphism.

We explain the different advantages of working with $QH^\ast (M;\mathbb{C})$ and $QH^\ast(M;\Lambda)$ as well as examples of cases where those advantages are used.

$\bullet$ \textit{The advantage of working with $QH^\ast (M;\mathbb{C})$:}
\begin{enumerate}
\item $QH^\ast (M;\mathbb{C})$ carries a $\mathbb{Z}$-grading while $QH^\ast(M;\Lambda)$ does not. Thus, to use spectral invariants it is preferable to work with $QH^\ast (M;\mathbb{C})$ than $QH^\ast(M;\Lambda)$ as the $\Z$-grading allows us to study both the action and the index of spectral invariants.

\begin{example}
Theorem \ref{my lemma}, which plays a crucial role in the first part of the proof of Theorem \ref{main theo}, is proven by using the information of both the action and the index of spectral invariants and thus, it is proven only in the setting where we have a $\mathbb{Z}$-grading of the quantum cohomology ring.
\end{example}

\item The algebraic structure of $QH^\ast (M;\mathbb{C})$ tends to be simpler than that of $QH^\ast(M;\Lambda)$. 

\begin{example} 
$QH^\ast (\mathbb{C}P^2;\mathbb{C})$ is a field and $QH^\ast(\mathbb{C}P^2;\Lambda)$ splits into a direct sum of three fields. $QH^\ast (S^2\times S^2 ; \C)$ splits into a direct sum of two fields and $QH^\ast(S^2\times S^2 ;\Lambda)$ splits into a direct sum of four fields. 

The quantum cohomology ring $QH^\ast (\mathbb{C}P^2;\mathbb{C})$ being a field has important consequences as pointed out in Remark \ref{remark after finite gamma} which do not follow only from semi-simplicity. This is precisely what we use in the proof of Theorem \ref{Q of PW}.

\end{example}
\end{enumerate}

$\bullet$ \textit{The advantage of working with $QH^\ast (M;\Lambda)$:} With $\Lambda$-coefficients, we have a very rich Lagrangian Floer theory developed by Fukaya--Oh--Ohta--Ono. In particular, the superpotential techniques are very useful to detect Lagrangian submanifolds that have non-trivial Floer coholomogy groups. 

\begin{example}
Finding certain Lagrangian submanifolds that have non-trivial Floer cohomology groups via superpotential techniques is a key step in the second part of the proof of Theorem \ref{main theo} explained in Section \ref{part2}.
\end{example}

To sum up, in the first part of the proof of Theorem \ref{main theo} (Section \ref{part1}), we need to work with $QH^\ast (M;\mathbb{C})$ while in the second part of the proof of Theorem \ref{main theo} (Section \ref{part2}), we greatly benefit from the advantage of working with $QH^\ast (M;\Lambda)$. In order to connect arguments in Part 1 and Part 2 which are done in different algebraic settings, we will need the following comparison between spectral invariants of a quantum cohomology class in $QH^\ast(M;\mathbb{C})$ and its embedded quantum cohomology class in $QH^\ast(M;\Lambda)$.

\begin{lemma}\label{rho compare}
Let $(M,\omega)$ be a monotone symplectic manifold. For any class $a\in QH^\ast(M;\mathbb{C})\backslash \{0\}$ and a Hamiltonian $H$, we have 
$$\rho(H,\iota(a)) = \rho(H,a).$$

\end{lemma}

The value $\rho(\cdot,a)$ denotes the spectral invariant of $a\in QH^\ast(M;\mathbb{C})$ while the value $\rho(\cdot,\iota(a))$ denotes the spectral invariant of its embedded element $\iota(a)\in QH^\ast(M;\Lambda)$. 

The following is a direct consequence of Lemma \ref{rho compare}.

\begin{lemma}\label{rho compare 2}
Let $e\in QH^0(M;\mathbb{C})$ be an idempotent. Assume that $e\cdot QH^{even}(M;\mathbb{C})$ is a field. Then, we have
$$\zeta_{\iota(e)}(\widetilde{\phi}) = \zeta_e (\widetilde{\phi})$$
for any $\widetilde{\phi} \in  \widetilde{\Ham}(M,\omega)$. In particular, 
$$\zeta_{\iota(e)}:\widetilde{\Ham}(M,\omega) \to \mathbb{R}$$
is a homogeneous quasimorphism.
\end{lemma}

\begin{remark}
A priori Lemma \ref{rho compare 2} is not obvious as we do not know if $\iota(e)$ is a unit of a field factor of $QH^\ast(M;\Lambda)$ i.e. $\iota(e) \cdot QH^\ast(M;\Lambda)$ is a field (Theorem \ref{EP quasimorphism}). For example, $QH^\ast(\C P^2;\C)$ is a field but $QH^\ast(\C P^2;\Lambda)$ splits into a direct sum of three fields and the identity element $1 \in QH^\ast(\C P^2;\C)$ embeds to $1_\Lambda \in QH^\ast(\C P^2;\Lambda)$ which is not an unit of a field factor.
\end{remark}

\begin{proof}[Proof of Lemma \ref{rho compare}]
 Similar technical results appeared in literature, e.g. \cite[Section 5.4]{[BC09]}, \cite[Propositions 2.21, 6.6]{[UZ16]}. Nevertheless, we give a proof for the sake of clarity. 
 
 We start by briefly recalling the construction of the Fukaya--Oh--Ohta--Ono type Floer chain complex from Chapter 2 in \cite{[FOOO19]}. First of all, we introduce the downward universal Novikov field
 $$\Lambda^{\downarrow}:=\{ \sum_{j\geq 0} a_j T^{\lambda_j}:a_j \in \mathbb{C},\lim_{j\to +\infty} \lambda_j=-\infty\}$$
 and recall that an element $\wt{z}=[z,w]$ of $\widetilde{\mathscr{P}}(H)$ is a capped periodic orbit of $H$, i.e. a pair of a periodic orbit $z$ and its capping $w$, as we have defined in Section \ref{preliminaries}. For a non-degenerate Hamiltonian $H$, define
\begin{equation}
    \widehat{CF}_{\ast}(H;\Lambda^{\downarrow}):= \{ \sum_{j\geq 0} a_j \cdot \wt{z}_j T^{\lambda_j}:a_j \in \mathbb{C}, \wt{z}_j \in \wt{\mathscr{P}}(H), \lim_{j\to +\infty} \lambda_j=-\infty\} .
\end{equation}

 We define the Floer chain complex by the following quotient
$$CF_\ast(H;\Lambda^{\downarrow}):=\widehat{CF}_\ast(H;\Lambda^{\downarrow}) / \sim$$
where the equivalence relation is defined by
$$[z,w] \sim [z',w']\otimes T^{\tau} \Longleftrightarrow z=z',\ \omega(w)=\omega(w')+\tau.$$
We describe a natural chain map from the Floer chain complex $CF_\ast(H)$, which was defined in Section \ref{HF}, to the Fukaya--Oh--Ohta--Ono type Floer chain complex $CF_\ast(H;\Lambda^{\downarrow})$:
$$j: CF_\ast(H) \hookrightarrow CF_\ast(H;\Lambda^{\downarrow}).$$
Consider
\begin{equation}
    \widehat{CF}_\ast(H;\mathbb{C}):= \{\sum_{k \leq k_0 , k_0 \in \Z} a_k \cdot \wt{z}_k s^{k}:a_k \in \mathbb{C}, \wt{z}_k \in \wt{\mathscr{P}}(H)\} 
\end{equation}
where $s$ is the formal variable used to define the field of Laurent series $\C [[s^{-1},s]$ that appear in the definition of the quantum homology in Section \ref{Quantum (co)homology}. Then, $CF_\ast(H)$ satisfies
$$CF_\ast(H)=\widehat{CF}_\ast(H;\mathbb{C}) / \sim$$
where
$$[z,w] \sim [z',w']\otimes s^k \Longleftrightarrow z=z',\ \omega(w)=\omega(w')+k\cdot \lambda_0.$$
It is easy to see that the inclusion 
\begin{equation}\label{j}
\begin{gathered}
    j: \C [[s^{-1},s] \hookrightarrow \Lambda^{\downarrow} \\
    s \mapsto  T^{\lambda_0}
\end{gathered}
\end{equation}
induces the inclusion
\begin{equation}
\begin{gathered}
    \hat{j}:\widehat{CF}_\ast (H;\mathbb{C})  \hookrightarrow \widehat{CF}_\ast(H;\Lambda^{\downarrow}) \\
    [z,w] \otimes s \mapsto [z,w]\otimes T^{\lambda_0}
\end{gathered}
\end{equation}
which induces the following chain map (by abuse of notation):
\begin{equation}\label{j2}
j:CF_\ast(H) \to CF_\ast(H;\Lambda^{\downarrow}).
\end{equation}

Now, notice that $\Lambda^{\downarrow}$ is a $\C [[s^{-1},s]$-module and by using the inclusion \ref{j}, it has the following split of $\C [[s^{-1},s]$-modules:
\begin{equation}
    \Lambda^{\downarrow} =  j(\C [[s^{-1},s]) \oplus j(\C [[s^{-1},s])^{\perp}
\end{equation}
where 
\begin{equation}
    j(\C [[s^{-1},s])^{\perp}:= \{\sum_{j\geq 0} a_j T^{\lambda_j}:a_j \in \mathbb{C},\lim_{j\to +\infty} \lambda_j=-\infty , \lambda_j \notin \lambda_0 \cdot \Z \}.
\end{equation}

Similarly, by considering the chain map \ref{j2}, the Floer chain complex splits as follows:
\begin{equation}\label{split}
    CF_\ast(H;\Lambda^{\downarrow})=j(CF_\ast (H) ) \oplus j(CF_\ast (H) )^{\perp}
\end{equation}
where
\begin{equation}
\begin{gathered}
    j(CF_\ast (H) )^{\perp}:= \{\sum_{j\geq 0} a_j\wt{z}_j  T^{\lambda_j}:a_j  \in \mathbb{C},\wt{z}_j \in \wt{\mathscr{P}}(H),\\
    \lim_{j\to +\infty} \lambda_j=-\infty , \lambda_j \notin \lambda_0 \cdot \Z \} .
\end{gathered}
\end{equation}

Now, we start the proof of Lemma \ref{rho compare}. We first prove $\rho(H,\iota(a)) \leq \rho(H,a)$ for any $a \in QH^\ast(M;\C) $ and a Hamiltonian $H$. By the continuity property of spectral invariants (Proposition \ref{prop spec inv} (1)), it is enough to prove the case where $H$ is non-degenerate. Let $a\in QH^\ast(M;\mathbb{C})\backslash \{0\}$ and $H$ be a non-degenerate Hamiltonian. The inclusion $j$ induces the following map on homology:
$$j_\ast : HF_\ast(H) \to HF_\ast(H;\Lambda^{\downarrow}).$$

Note that
$$PSS_{H,\Lambda}\circ PD \circ \iota = j_\ast \circ PSS_H\circ PD$$
where $PSS_H$ on the right hand side denotes the PSS-isomorphism
$$PSS_H : QH_\ast (M;\mathbb{C}) \xrightarrow{\sim} HF_\ast (H)$$
while $PSS_H$ on the left hand side denotes the PSS-isomorphism
$$PSS_{H,\Lambda^{\downarrow}}  : QH_\ast(M;\Lambda^{\downarrow}) \xrightarrow{\sim} HF_\ast(H;\Lambda^{\downarrow})$$
and $PD$ denotes the Poincar\'e duality between quantum homology and quantum cohomology. Consider the diagram
\[
  \begin{CD}
     HF_\ast ^{\tau}(H) @>{i^\tau _\ast}>>  HF_\ast (H)  @<PSS_H \circ PD<<  QH^\ast(M;\mathbb{C})    \\
    @VV{j_\ast}V     @VV{j_\ast}V  @VV{\iota}V\\
     HF_\ast ^{\tau}(H;\Lambda^{\downarrow}) @>i^\tau _\ast>>  HF_\ast (H;\Lambda^{\downarrow}) @<PSS_{H,\Lambda^{\downarrow}}\circ PD<<  QH^\ast(M;\Lambda) 
  \end{CD}
\]

As $j_\ast$ preserves the action filtration, the diagram commutes and for tautological reasons, we get 
\begin{equation}\label{leq}
    \rho(H,\iota(a)) \leq \rho(H,a).
\end{equation}

We next show 
$$\rho(H,\iota(a)) \geq \rho(H,a)$$
for any $a \in QH^\ast(M;\C) $ and a Hamiltonian $H$. We prove this inequality for ``nice'' Hamiltonians where a ``nice'' Hamiltonian $H$ has the properties that it is non-degenerate and the action functional $\mathscr{A}_H$ induces a bijection between $\widetilde{\mathscr{P}} (H) $ and $\Spec (H)$. As one can approximate any Hamiltonian with a sequence of ``nice'' Hamiltonians, restricting our focus to this class of Hamiltonians is enough. We argue by contradiction: assume there is a ``nice'' Hamiltonian $H$ and a class $a \in QH^\ast (M ; \C)$ such that 
\begin{equation}\label{contradiction}
    \rho(H,\iota(a)) < \rho(H,a) .
\end{equation}
There exist Floer cycles $\alpha \in CF_\ast (H)$, $\alpha' \in CF_\ast (H; \Lambda^{\downarrow})$ such that
\begin{subequations}
\begin{gather}
    \rho(H,a)= \mathscr{A}_{H} (\alpha), \\
[\alpha]=PSS_{H} \circ PD (a)
\end{gather}
\end{subequations}
and
\begin{subequations}
\begin{gather}
    \rho(H,\iota(a))= \mathscr{A}_{H} (\alpha'),\\
[\alpha']=PSS_{H,\Lambda^{\downarrow}} \circ PD (\iota(a)).
\end{gather}
\end{subequations}

As $[j(\alpha)]=\iota([\alpha])=[\alpha']$, there exists $\beta \in CF_\ast (H;\Lambda^{\downarrow})$ such that
$$\alpha' = j(\alpha) + \partial (\beta).$$
First of all, as $\rho(H , a) \in \Spec (H)$ and $H$ is ``nice'', there exists the ``action carrier'' $\widetilde{z} \in \widetilde{\mathscr{P}} (H)$ such that 
\begin{subequations}
\begin{gather}
    \mathscr{A}_{H} (\widetilde{z})=\mathscr{A}_{H} (\alpha)\label{alpha 1},\\
        \alpha = \lambda \cdot \widetilde{z} + low \label{alpha 2}
\end{gather}
\end{subequations}

for some $\lambda \in \C \backslash \{0\}$ where $low$ denotes some chain which satisfies 
$$\mathscr{A}_{H} (low ) <\mathscr{A}_{H} (\widetilde{z})=\mathscr{A}_{H} (\alpha) . $$
We will repeatedly use this convenient notation in the sequel analogously. Next, as we have
$$\mathscr{A}_{H} (j(\alpha) + \partial (\beta))=\mathscr{A}_{H} (\alpha') < \mathscr{A}_{H} (j(\alpha))=\mathscr{A}_{H} (\alpha) $$
from the assumption \ref{contradiction}, equations \ref{alpha 1} and \ref{alpha 2} imply the following about $\partial (\beta)$:
\begin{equation}\label{del beta}
    \partial (\beta)= -\lambda \cdot j(\widetilde{z}) + low.
\end{equation}
Now, we decompose $\beta \in CF_\ast (H; \Lambda^{\downarrow})$ with respect to the split \ref{split}:
\begin{equation}
     \beta= j(\beta_1) + \beta_2
\end{equation}
where $\beta_1 \in CF_{\ast} (H), \beta_2 \in j(CF_{\ast} (H)) ^{\perp}.$ The Floer boundary map preserves the split \ref{split}, so $\partial(\beta)$ splits as follows:
\begin{equation}\label{del beta 2}
    \partial(\beta)= \partial(j(\beta_1)) + \partial(\beta_2).
\end{equation}
By comparing equations \ref{del beta} and \ref{del beta 2}, as $-\lambda \cdot j(\widetilde{z}) \in j(CF_{\ast}(H))$, we see that $-\lambda \cdot j(\widetilde{z})$ is contained in $\partial (j(\beta_1))$ and not in $\partial(\beta_2)$:
\begin{equation}
    j( \partial (\beta_1)) =\partial (j(\beta_1))= -\lambda \cdot j(\widetilde{z}) + low
\end{equation}
Consider the Floer cycle 
$$ \alpha + \partial (\beta_1) \in CF_{\ast} (H).$$
This satisfies 
\begin{subequations}
\begin{gather}
    [\alpha + \partial (\beta_1)]= [\alpha]=PSS_{H} \circ PD (a) , \label{class} \\
\mathscr{A}_{H} (\alpha + \partial (\beta_1)) < \mathscr{A}_{H} (\alpha)= \rho (H , a).\label{ineq}
\end{gather}
\end{subequations}
Note that inequality \ref{ineq} follows from cancellation of the action carriers $\widetilde{z}$ of $\alpha$ and $\partial (\beta_1)$. The relations \ref{class} and \ref{ineq} contradict the definition of $\rho(H,a)$. This completes the proof of 
\begin{equation}\label{geq}
    \rho(H,\iota(a)) \geq \rho(H,a).
\end{equation}
Inequalities \ref{leq} and \ref{geq} imply Lemma \ref{rho compare}.
\end{proof}

We obtain the following from Lemma \ref{rho compare 2}.

\begin{corol}\label{h to sh}
Let $(M,\omega)$ be a monotone symplectic manifold. Assume that $e\in QH^0(M;\mathbb{C})$ is an idempotent and $e\cdot QH^{even}(M;\mathbb{C})$ is a field. If a subset $S \subset M$ is $\iota(e)$-heavy, then $S$ is $e$-superheavy. 
\end{corol}

\begin{proof}[Proof of Corollary \ref{h to sh}]

Lemma \ref{rho compare 2} implies that $S$ is $e$-heavy. However, as $e\in QH^0(M;\mathbb{C})$ is a unit of a field factor of $QH^{even}(M;\mathbb{C})$, $\zeta_e$ is a {\qmor} so $S$ is $e$-superheavy.
\end{proof}

\subsection{Proof of Theorem \ref{main theo}--Part 2}\label{part2}

In this subsection, we prove Theorem \ref{theo part2} which was used to complete the proof of Theorem \ref{main theo} in the end of Section \ref{part1}.

\begin{proof}[Proof of Theorem \ref{theo part2}]
We argue the cases $n=2$ and $n=4$ separately.

$\bullet$ \textit{Case $n=2$}: In this case, $\zeta_{e_+} \neq \zeta_{e_-}$ was already proven by Eliashberg--Polterovich in \cite{[EliP10]} by an approach different to what we discuss in this section. In this section, we will prove $\zeta_{e_+} \neq \zeta_{e_-}$ by using the following result of Fukaya--Oh--Ohta--Ono \cite{[FOOO12]}, \cite{[FOOO19]}. The same argument will be used in the case where $n=4$.

\begin{theo}\label{FOOO S2S2}$($\cite[Lemma 23.3(1), Lemma 23.5]{[FOOO12]}$)$
\begin{enumerate}
\item In $S^2\times S^2$, there exists a monotone Lagrangian submanifold $L_0$ diffeomorphic to $T^2$ such that
$$HF((L_0,b_0);\Lambda)\neq 0$$
for a certain bounding cochain $b_0 \in H^1(L_0;\Lambda_0)/H^1(L_0;2\pi i\mathbb{Z}).$

\item The anti-diagonal in $S^2\times S^2$ denoted by $L_1$ is unobstructed and satisfies 
$$HF(L_1;\Lambda)\neq 0.$$
\item $L_0$ and $L_1$ are disjoint:
$$L_0 \cap L_1=\emptyset.$$

\end{enumerate}
\end{theo}

Now, consider the natural embedding
$$\iota: QH^\ast(Q^2;\mathbb{C}) \hookrightarrow QH^\ast(Q^2;\Lambda).$$
As the closed-open map maps the identity element of the quantum cohomology ring to the identity element of the Lagrangian Floer cohomology group, we have 
\begin{equation}
    \begin{gathered}
        \mathcal{CO}^0 _b(1) = \PD([L_0]) \neq 0 \in HF^\ast ((L_0, b) ; \Lambda) \\
  \mathcal{CO}^0 (1)=\PD([L_1]) \neq 0 \in HF^\ast (L_1;\Lambda).
    \end{gathered}
\end{equation}
Since
$$1=\iota(e_+)+\iota(e_-),$$
it is 
$$\mathcal{CO}^0 _b(\iota(e_+))\neq 0 \ \ or\ \ \mathcal{CO}^0 _b(\iota(e_-))\neq 0. $$
As $\iota(e_\pm)$ are both idempotents, by Theorem \ref{co map}, we deduce that $L_0$ is at least either $\iota(e_+)$-heavy or $\iota(e_-)$-heavy. Corollary \ref{h to sh} implies that $L_0$ is at least either $e_+$-superheavy or $e_-$-superheavy. Next, by looking at the second equation, the same argument implies that $L_1$ is at least either $e_+$-superheavy or $e_-$-superheavy. As $L_0$ and $L_1$ are disjoint, Proposition \ref{disjoint-heavy} implies that they cannot be both $e_+$-superheavy or both $e_-$-superheavy at once. This implies
$$\zeta_{e_+} \neq \zeta_{e_-}.$$

\begin{remark}
From this argument, it follows that either
\begin{itemize}
\item $L_0$ is $e_+$-superheavy and $L_1$ is $e_-$-superheavy 
\item $L_0$ is $e_-$-superheavy and $L_1$ is $e_+$-superheavy 
\end{itemize}
but it is not clear which one of the two actually holds. Eliashberg--Polterovich's approach shows that the former holds \cite{[EliP10]}. 
\end{remark}

$\bullet$\textit{Case $n=4$}: The key of the proof is to find two disjoint Lagrangian submanifolds in $Q^4$ having non-vanishing Floer cohomology just as in the previous case. We use results of Nishinou--Nohara--Ueda and Nohara--Ueda which we will now briefly explain.

The relation between the superpotential and Lagrangian Floer cohomology has been studied extensively. After a pioneering work of Cho \cite{[Cho04]}, Fukaya--Oh--Ohta--Ono computed the superpotential for toric symplectic manifolds in \cite{[FOOO10]}. Later, Nishinou--Nohara--Ueda computed the superpotential for symplectic manifolds admitting a toric degeneration in \cite{[NNU10]}. This lead Nohara--Ueda to study the Floer cohomology of non-torus fibers in partial flag manifolds in \cite{[NU16]}. We state some of their results which will be relevant for us.

\begin{theo}\label{ueda}$($\cite[Theorem 10.1, Section 11]{[NNU10]}, \cite[Theorem 1.2, Example 3.3]{[NU16]}$)$

Let $\Phi:Gr_{\mathbb{C}}(2,4) \to \mathbb{R}^4$ be the Gelfand--Cetlin system with the Gelfand--Cetlin polytope $\Delta:=\Phi(Gr_{\mathbb{C}}(2,4))$. Denote the fiber of $u \in \Delta$ by $L(u)$:
$$L(u):=\Phi^{-1}(u).$$
We identify $Gr_{\mathbb{C}}(2,4)$ with the adjoint orbit of 
$$\lambda=\diag (4,4,0,0)$$
so that it is monotone.

\begin{enumerate}
\item For
$$u_0:=( 2, 3,1,2) \in \Int(\Delta),$$
there exists a bounding cochain $b\in H^1(L(u_0);\Lambda)$ such that 
$$HF((L(u_0),b); \Lambda) \simeq QH^\ast (T^4 ;\Lambda).$$

\item There exists $u_1 \in \partial \Delta$ such that its fiber $L(u_1)$ is Lagrangian and diffeomorphic to $U(2)\simeq S^1\times S^3$ with non-trivial Floer cohomology:
$$HF((L(u_1),\pm \pi i/2\cdot e_1);\Lambda) \simeq QH^\ast (S^1\times S^3;\Lambda)$$
for a bounding cochain $b=\pm \pi i/2\cdot e_1$ where $e_1$ is the generator of $H^1(L(u_1);\mathbb{Z})$.

\end{enumerate}
\end{theo}

\begin{remark}
\begin{enumerate}
\item $Q^4$ is symplectomorphic to $Gr_{\mathbb{C}}(2,4)$. 

\item Theorem \ref{ueda} (1) was proven for any identification of $Gr_{\mathbb{C}}(2,4)$ with $\mathcal{O}_\lambda$ where
$$\lambda=\diag (2\alpha,2\alpha,0,0),$$
$$u_0:=( \alpha, 3\alpha/2,\alpha/2,\alpha) \in \Int(\Delta)$$
for any $\alpha>0$. If we choose $\alpha=2$, the Kirillov-Konstant form $\omega_\lambda$ defines a monotone symplectic form by the monotonicity criterion in Section \ref{GC system}.

\item Note that $L(u_0)\simeq  T^4$ and $L(u_1)\simeq U(2) \simeq S^1\times S^3$ are both monotone since they are both located in the center of a Lagrangian facet of the Gelfand--Cetlin polytope. This follows from a result of Yunhyung Cho and Yoosik Kim \cite{[CK19]} where they classify monotone fibers of Gelfand--Cetlin polytopes.
\end{enumerate}
\end{remark}

Let $L(u_0), L(u_1)$ be as in Theorem \ref{ueda}. We argue exactly as in the case where $n=2$. As the closed-open map maps the identity element of the quantum cohomology ring to the identity element of the Lagrangian Floer cohomology group, we have 
\begin{subequations}
\begin{gather}
    \mathcal{CO}^0 _{b}(1)=\PD([L(u_0)]) \neq 0 \in HF^\ast ((L(u_0), b);\Lambda), \\
 \mathcal{CO}^0 _{\pm \pi i/2\cdot e_1}(1)=\PD([L(u_1)]) \neq 0 \in HF^\ast ((L(u_1), \pm \pi i/2\cdot e_1);\Lambda).
\end{gather}
\end{subequations}

Since
$$1=\iota(e_+)+\iota(e_-),$$
the first equation and Theorem \ref{co map} imply that $L(u_0)$ is $e_+$-superheavy or $e_-$-superheavy. We have used that by Corollary \ref{h to sh}, $\iota(e_\pm)$-heaviness is equivalent to $e_\pm$-superheaviness. Next, by looking at the second equation, the same argument implies that $L(u_1)$ is $e_+$-superheavy or $e_-$-superheavy. As $L(u_0)$ and $L(u_1)$ are disjoint (recall that they are fibers of distinct points in the Gelfand--Cetlin polytope), we conclude that they cannot be both $e_+$-superheavy or both $e_-$-superheavy at once. This implies
$$\zeta_{e_+} \neq \zeta_{e_-}.$$
\end{proof}

\subsection{Generalization of Theorem \ref{main theo}}\label{Proof of main theo general}

In this section, we prove the following slight generalization of Theorem \ref{main theo}.

\begin{theo}\label{main theo general}
Let $(M,\omega)$ be a monotone symplectic manifold (with the same monotonicity constant as $Q^n,\ n=2,4)$ such that $QH^\ast(M;\mathbb{C})$ is semi-simple. Assume that there exists a Lagrangian submanifold $L$ of $(M,\omega)$ such that $HF((L,b);\Lambda)\neq 0$ for some bounding cochain $b$. Then, there exists a non-trivial homogeneous quasimorphism
$$\mu: \Ham( Q^n\times M) \to \mathbb{R}$$
which is both $C^0$-continuous and Hofer Lipschitz continuous where $Q^n\times M\ (n=2,4)$ denotes the monotone product.
\end{theo}

\begin{remark}
 The existence of a homogeneous quasimorphism on $\Ham(Q^n\times M),\ n=2,4$ (instead of on the universal cover) where $(M,\omega)$ is as in Theorem \ref{main theo general} was not known to the best of our knowledge. Note that examples of $(M,\omega)$ which satisfy the assumptions in Theorem \ref{main theo general} include $\mathbb{C}P^n$, 1, 2 and 3 point monotone blow-ups of $\mathbb{C}P^2$, $Q^n$ and their monotone products.

\end{remark}

We start with some preliminary results on the product of semi-simple algebras.

Let $(M_j,\omega_j)\ (j=1,2)$ be monotone symplectic manifolds. Denote the generators of $\pi_2(M_j)/ \Ker(\omega_j)$ by $s_j$ which satisfy
$$\omega_j(s_j)=\lambda_{M_j},\ \ c_1(TM_j)(s_j)=N_{M_j}$$
where $\lambda_{M_j}$ denotes the rationality constant and $N_{M_j}$ denotes the minimal Chern number of $(M_j,\omega_j)$.

In the case where the monotonicity constants of $(M_j,\omega_j)\ (j=1,2)$ coincide, one can consider their product $(M_1\times M_2,\omega_1 \oplus \omega_2)$ which is also a monotone symplectic manifold. It has the same monotonicity constant as $(M_j,\omega_j)\ (j=1,2)$ and its minimal Chern number $N_{M_1\times M_2}$ is the greatest common divisor of $N_{M_1}$ and $N_{M_2}$. As above, we denote the generator of the $\pi_2(M_1\times M_2)/ \Ker(\omega_1 \oplus \omega_2)$ by $s$ which satisfies
$$(\omega_1 \oplus \omega_2)(s)=\lambda_{M_1\times M_2},\ \ c_1(T(M_1\times M_2))(s)=N_{M_1\times M_2}.$$

Entov--Polterovich proved the following in \cite{[EP08]}.

\begin{theo}$($\cite[Theorem 5.1, Theorem 6.1]{[EP08]}$)$

Let $(M_j,\omega_j)\ (j=1,2)$ be monotone symplectic manifolds. Assume that their quantum homology rings 
$$QH_{even} (M_j;\mathbb{C})=H_{even}(M_j;\mathbb{C})\otimes \mathbb{C}[|s_j ^{-1},s_j]$$
are both semi-simple and that at least one of $M_j,\ j=1,2$ satisfies $H_{2k-1}(M_j;\mathbb{C})=0$ for all $k \in \mathbb{Z}$. Then, 
$$QH_{even} (M_1\times M_2;\mathbb{C})=H_{even}(M_1 \times M_2;\mathbb{C} ) \otimes \mathbb{C}[|s ^{-1},s]$$ is semi-simple.
\end{theo}

One can consider the following embedding:
$$\sigma : QH_\ast (M_1;\mathbb{C}) \hookrightarrow QH_\ast (M_1 \times M_2; \mathbb{C})$$
$$a\cdot s_1 \mapsto a\otimes [M_2] \cdot s^{ {N_{M_1}}/ N_{M_1\times M_2} }.$$
Of course, one can consider an analogous embedding for $M_2$.

We are now ready to prove Theorem \ref{main theo general}. We will use the cohomological counterpart of the results above.

\begin{proof}[Proof of Theorem \ref{main theo general}]

As $QH^\ast (M;\mathbb{C})$ is semi-simple, it splits into a direct sum of fields $\{Q_j\}$:
$$QH^\ast (M;\mathbb{C})=Q_1 \oplus Q_2 \oplus \cdots \oplus Q_l.$$
We decompose the identity element $1_M \in QH^\ast (M;\mathbb{C})$ with respect to this decomposition:
$$1_M=e_1+e_2+\cdots +e_l$$
where $e_j $ is a unit of $Q_j$. As we argued in the proof of Theorem \ref{main theo}, as
$$\mathcal{CO}_b ^0(1_M)=\mathcal{CO}_b ^0(\iota(e_1))+\mathcal{CO}_b ^0(\iota(e_2))+\cdots +\mathcal{CO}_b ^0(\iota(e_l)) \neq 0,$$
Proposition \ref{co map} implies that $L$ is $\iota(e_j)$-heavy for some $j\in \{1,2,\cdots, l\}$. Without loss of generality, we assume $j=1$. Moreover, Corollary \ref{h to sh} implies that $L$ is actually $e_1$-superheavy.

Recall that in the proof of Theorem \ref{main theo}, we have seen that in $Q^n,\ n=2,4$, there exist two disjoint Lagrangian submanifolds $L_0$ and $L_1$ which satisfy either one of the following:
\begin{enumerate}
\item $L_0$ is $e_+$-superheavy and $L_1$ is $e_-$-superheavy.
\item $L_0$ is $e_-$-superheavy and $L_1$ is $e_+$-superheavy.
\end{enumerate}
Without loss of generality, we assume the former. By \cite[Theorem 1.7]{[EP09]}, $L_0\times L$ is $e_+ \otimes e_1$-superheavy and $L_1 \times L$ is $e_- \otimes e_1$-superheavy. 

Now, as $QH^\ast (Q^n\times M;\mathbb{C})$ is also semi-simple, we consider its decomposition into fields and the decomposition of the identity element $1_{Q^n\times M} \in QH^\ast (Q^n\times M;\mathbb{C})$ with respect to this split:
$$QH^\ast (Q^n\times M;\mathbb{C})=Q'_1 \oplus Q'_2 \oplus \cdots \oplus Q'_{l'}$$
$$1_{Q^n\times M}=e'_1+e'_2+\cdots +e'_{l'}$$
for some $l' \in \mathbb{N}$. As $e_+ \otimes e_1$ and $e_- \otimes e_1$ are idempotents of $QH^\ast (Q^n\times M;\mathbb{C})$, by \cite[Theorem 1.5 (3)]{[EP09]}, there exist $j_0,j_1 \in \{1,2,\cdots , l'\}$ such that $L_0\times L$ is $e'_{j_0}$-heavy and $L_1 \times L$ is $e'_{j_1}$-heavy. For $\zeta_{e'_{j_0}}, \zeta_{e'_{j_1}}: \widetilde{\Ham}(Q^n\times M) \to \mathbb{R}$ both being homogeneous quasimorphisms, heaviness and superheaviness are equivalent for $e'_{j_0}$ and $e'_{j_1}$, thus $L_0\times L$ is $e'_{j_0}$-superheavy and $L_1 \times L$ is $e'_{j_1}$-superheavy. As $L_0\times L$ and $L_1\times L$ are disjoint, $L_1\times L$ is not $e'_{j_0}$-superheavy. This implies 
$$\zeta_{e'_{j_0}} \neq \zeta_{e'_{j_1}}.$$
Thus, it follows from Theorem \ref{general} that 
$$\mu:=\zeta_{e'_{j_0}}-\zeta_{e'_{j_1}}$$
defines a non-trivial homogeneous quasimorphism 
$$\mu: \Ham( Q^n\times M) \to \mathbb{R}$$
 which is both $C^0$-continuous and Hofer Lipschitz continuous.

\end{proof}

\subsection{Proof of Theorem \ref{Q of PW}}\label{proof of Q of PW}

In this subsection, we precisely state the question of Polterovich--Wu which appeared in Section \ref{A question of Polterovich--Wu} and prove Theorem \ref{Q of PW} as an application of Lemma \ref{rho compare 2}.

According to a computation due to Wu \cite{[Wu15]}, $QH^\ast(\mathbb{C}P^2;\Lambda)$ is semi-simple and splits into a direct sum of three fields:
$$QH^\ast(\mathbb{C}P^2;\Lambda)=Q_1\oplus Q_2 \oplus Q_3.$$
We denote the corresponding split of the identity element $1_{\Lambda} \in QH^\ast(\mathbb{C}P^2;\Lambda)$ as follows:
$$1_{\Lambda} =e_1 +e_2+e_3$$
where $\{e_j\}_{j=1,2,3}$ are 
$$e_j:=\frac{1}{3} (1_{\Lambda} + \theta^{j} u T^{-\frac{1}{3}\lambda_0} +  \theta^{2j} u^2 T^{-\frac{2}{3}\lambda_0}),$$
$u$ is the generator of $H^2(\mathbb{C}P^2;\mathbb{C})$ and 
$$\lambda_0:=\langle \omega_{FS}, [\mathbb{C}P^1] \rangle,\ \theta:=e^{\frac{2\pi i}{3}}.$$
Note that $u$ satisfies
$$u^3=T^{\lambda_0}.$$ 
These idempotents give rise to three homogenous quasimorphisms (or symplectic quasi-states) $\{\zeta_{e_j}\}_{j=1,2,3}$:

$$\zeta_{e_j}: \widetilde{\Ham}(\mathbb{C}P^2) \to \mathbb{R}$$
$$\zeta_{e_j}(\widetilde{\phi}):= \lim_{k \to +\infty} \frac{\rho(\widetilde{\phi} ^k,e_j)}{k}$$
for each $j=1,2,3$.

\begin{remark}
It will not be used in this paper but we point out that $\zeta_{e_j}$ descends to $\Ham(\mathbb{C}P^2)$ as $\pi_1(\Ham(\mathbb{C}P^2))=\mathbb{Z}_3$.
\end{remark}

Polterovich posed the following question:

\begin{question}\label{precise Q of PW}$($\cite[Remark 5.2]{[Wu15]}$)$

Is it possible to distinguish the symplectic quasi-states/morphisms for the three idempotents of $QH^\ast(\mathbb{C}P^2;\Lambda)$?
\end{question}

Note that $\zeta_j$ which appeared in the statement of this question in Section \ref{A question of Polterovich--Wu} is precisely $\zeta_{e_j}$ defined above. We now prove Theorem \ref{Q of PW} which answers this question in the negative.

\begin{proof}[Proof of Theorem \ref{Q of PW}]

We will show that 
$$\zeta_{e_j}= \zeta_{1_{\Lambda}}$$
for all $j=1,2,3$ where $1_{\Lambda} \in QH^\ast(\mathbb{C}P^2;\Lambda)$. By the triangle inequality, we have
$$\rho(\widetilde{\phi}^k,e_j) \leq \rho(\widetilde{\phi}^k,1_{\Lambda}) + \nu(e_j)$$
for any $k \in \mathbb{N}$. Thus,
\begin{equation}\label{ineq PW}
    \zeta_{e_j} \leq \zeta_{1_{\Lambda}}
\end{equation}
where 
$$\zeta_{1_{\Lambda}}(\widetilde{\phi}):=\lim_{k \to +\infty} \frac{\rho(\widetilde{\phi}^k,1_{\Lambda})}{k}$$
for $\widetilde{\phi} \in  \widetilde{\Ham}(\mathbb{C}P^2)$. As $QH^\ast (\mathbb{C}P^2;\mathbb{C})$ is a field, by Lemma \ref{rho compare 2}, we see that
\begin{equation}\label{PW1}
    \zeta_{1_{\Lambda}}=\zeta_{1}
\end{equation}
are both homogeneous quasimorphisms where $1 \in QH^\ast (\mathbb{C}P^2;\mathbb{C})$. Thus, the inequality \ref{ineq PW} and the homogeneity of $\zeta_{e_j}$ and $\zeta_{1_{\Lambda}}$ imply
\begin{equation}\label{PW2}
    \zeta_{e_j} = \zeta_{1_{\Lambda}}.
\end{equation}
Thus, putting equalities \ref{PW1} and \ref{PW2} together, we have proven
$$\zeta_{e_1}=\zeta_{e_2}=\zeta_{e_3}=\zeta_{1_{\Lambda}}=\zeta_{1}.$$
\end{proof}

\begin{remark}
A similar argument is applicable to the case where $M:=S^2 \times S^2$: As we have seen in Section \ref{QH of quadrics}, $QH^\ast(S^2\times S^2;\mathbb{C})$ splits into a direct sum of two fields. On the other hand, Fukaya--Oh--Ohta--Ono have computed in the proof of \cite[Theorem 23.4]{[FOOO19]}, $QH^\ast(S^2\times S^2;\Lambda)$ splits into a direct sum of four fields. Denote the two units of field factors of $QH^\ast(S^2\times S^2;\mathbb{C})$ by $e_\pm$ which satisfy
$$\PD(e_\pm)=\frac{[M]\pm [pt \times pt]s}{2}.$$
Denote the four units of the field factors of $QH^\ast(S^2\times S^2;\Lambda)$ by $e_{(\pm,\pm)}$ which satisfy
$$\PD(e_{(+,\pm)})=\frac{[M] + P\cdot T^{\lambda_0}}{4} \pm \frac{(A+B)T^{\lambda_0/2}}{4},$$
$$\PD(e_{(-,\pm)})=\frac{[M] - P\cdot T^{\lambda_0}}{4} \pm \frac{(A-B)T^{\lambda_0/2}}{4}$$
where
$$[M]:=[S^2\times S^2],\ \ P:=[pt\times pt],$$
$$A:=[S^2\times pt],\ \ B:=[pt\times S^2].$$
By using
$$\iota(e_+)= e_{(+,+)} + e_{(+, -)},$$
$$\iota(e_-)= e_{(-,+)} + e_{(-, -)},$$
we obtain 
$$\zeta_{e_+}=\zeta_{\iota(e_+)}=\zeta_{e_{(+,+)}}=\zeta_{e_{(+, -)}},$$
$$ \zeta_{e_-}=\zeta_{\iota(e_-)}=\zeta_{e_{(-,+)}}=\zeta_{e_{(-, -)}}.$$

\end{remark}

\subsection{Results on Lagrangian intersections}\label{lag int}

In this section, we discuss consequences of the proof of Theorem \ref{main theo} for Lagrangian intersections.

In proving Theorem \ref{main theo}, detecting disjoint Lagrangian submanifolds whose Floer cohomology is non-trivial is a crucial step which we discussed in Section \ref{part1}. As a by-product, we obtain certain results on Lagrangian intersections. 

A closed Lagrangian submanifold $L$ is called monotone if it satisfies
$$\omega|_{\pi_2(M,L)}=\lambda \cdot \mu|_{\pi_2(M,L)}$$
for some $\lambda>0$ where $\mu=\mu_L$ denotes the Maslov class. The minimal Maslov number $N_L$ is the positive generator of $ \mu(\pi_2(M,L))$ i.e. $ \mu(\pi_2(M,L))=N_L \mathbb{Z}.$ Recall that $\Lambda$ denotes the universal Novikov field 
$$\Lambda=\{\sum_{j=1} ^{\infty} a_j T^{\lambda_j} :a_j \in \mathbb{C}, \lambda _j  \in \mathbb{R},\lim_{j\to +\infty} \lambda_j =+\infty \}.$$

All the Lagrangian submanifolds concerned in the following are assumed to be oriented and relatively spin (for its definition, see Section \ref{deformed HF}). The statements in this section include the notion of deformed Floer cohomology defined by Fukaya--Oh--Ohta--Ono \cite{[FOOO09]}. For a quick review, see Section \ref{deformed HF}.

The main statement for Lagrangian intersection is the following.

\begin{theo}\label{Q4 intersection}
In $Q^n\ (n=2,4)$, there exist two monotone Lagrangian submanifolds $L_0,\ L_1$ that satisfy the following:

\begin{enumerate}
\item $L_0$ and $L_1$ are respectively diffeomorphic to
\begin{itemize}
\item $T^2$ and $S^2$ when $n=2$. 
\item $T^4$ and $S^1\times S^3$ when $n=4$. 
\end{itemize} 
\item $L_0$ and $L_1$ are disjoint.
\item Let $L$ be a Lagrangian submanifold in $Q^n$ which is 
\begin{itemize}
\item oriented when $n=2$. 
\item oriented and relatively spin when $n=4$. 
\end{itemize}
If $L$ is disjoint from both $L_0$ and $L_1$ i.e. if 
$$L \cap (L_0 \cup L_1) =\emptyset $$
then
$$HF((L,b);\Lambda)=0$$
for any bounding cochain $b$.
\end{enumerate}
\end{theo}

\begin{remark}
\begin{enumerate}
\item Under the symplectomophism between $Q^2$ and $S^2\times S^2$, the Lagrangian submanifolds $L_0$ and $L_1$ in Theorem \ref{Q4 intersection} correspond respectively to the so-called exotic torus defined by
$$\{(x,y) \in S^2\times S^2: x_1y_1+x_2y_2+x_3y_3=-1/2,x_3+y_3=0\}$$
which was studied in \cite{[EP09]}, \cite{[FOOO12]} and the anti-diagonal 
$$\{(x,y) \in S^2\times S^2:x=-y\}.$$
\item For more information about the two Lagrangian submanifolds in Theorem \ref{Q4 intersection}, see Theorems \ref{FOOO S2S2}, \ref{ueda} and related references.
\end{enumerate}
\end{remark}

For example, Theorem \ref{Q4 intersection} can be applied to the following two well-known Lagrangians in $Q^2$ and $Q^4$. In $Q^2$, there is a {\lag} torus $T$ which corresponds to the product of equatorial circles $S^1\times S^1$ in $S^2\times S^2$ under the symplectomophism between $Q^2$ and $S^2\times S^2$. In $Q^4$, there is the standard Lagrangian sphere $S^4$ which appears as the real locus 
$$S^4 = \{(x_0:\cdots: x_4) \in \C P^{5} : x_0 ^2 +\cdots +x_{3} ^2 = x_{4} ^2,\ x_j \in \R, j=0,\cdots,4 \}.$$

These Lagrangians $T$ and $S^4$ are known to have non-trivial Floer cohomology groups
$$HF(T;\Lambda)\neq 0 ,\ HF(S^4;\Lambda)\neq 0.$$

Theorem \ref{Q4 intersection} directly implies the following.

\begin{corol}
Any Hamiltonian deformation of $T$ in $Q^2$ or the standard Lagrangian sphere $S^4$ in $Q^4$ intersects either one of $L_0$ or $L_1$ in Theorem \ref{Q4 intersection}:

For any $\phi \in \Ham(Q^2),$
$$L_0 \cap \phi(T) \neq \emptyset \ \ \text{or}\ \ L_1 \cap \phi(T) \neq \emptyset.$$

For any $\phi \in \Ham(Q^4),$
$$L_0 \cap \phi(S^4) \neq \emptyset \ \ \text{or}\ \ L_1 \cap \phi(S^4) \neq \emptyset.$$
\end{corol}

\begin{remark}
In Theorem \ref{Q4 intersection}, it is crucial that we consider Floer cohomology without bulk-deformation. As it was studied by Fukaya--Oh--Ohta--Ono \cite{[FOOO12]} and Cho--Kim--Oh \cite{[CKO18]}, there exist Lagrangians in $Q^n\ (n=2,4)$ intersecting neither $L_0$ nor $L_1$ that have non-trivial bulk-deformed Floer cohomology.
\end{remark}

There are several ways to construct monotone Lagrangian submanifolds in $Q^n$ such as the Albers--Frauenfelder-type construction \cite{[AF08]} and the Biran-type construction \cite{[B01]}, \cite{[B06]}. Their precise
constructions and the relations among them are explained in \cite{[OU16]}. In particular, Oakley--Usher constructs monotone Lagrangian submanifolds in $Q^4$ which are diffeomorphic to $S^1\times S^3$ by these methods in \cite[Section 1.2]{[OU16]} denoted by $L^{Q} _{0,3}$ and $\mathbb{S}^{Q} _{0,3}$, which turn out to be Hamiltonian isotopic (see \cite[Theorem 1.4]{[OU16]}). However, the monotone Lagrangian submanifold $L_1$ in $Q^4$ which appeared in Theorem \ref{Q4 intersection} is not Hamiltonian isotopic to these examples due to Oakley--Usher as $L_1$ has minimal Maslov number $4$ (see \cite[Section 4.4]{[NU16]}) and Oakley--Usher's Lagrangian submanifold has minimal Maslov number $2$. Thus, we have the following.

\begin{prop}
The $4$-quadric $Q^4$ has two monotone Lagrangian submanifolds diffeomorphic to $S^1\times S^3$ which are not Hamiltonian isotopic.
\end{prop}

Basically, Theorem \ref{Q4 intersection} comes from the fact that the quantum cohomology ring $QH^\ast(Q^n;\mathbb{C})=H^\ast(Q^n;\mathbb{C}) \otimes_{\mathbb{C}} \mathbb{C}[t^{-1},t]]$ splits into a direct sum of two fields. In the case where the quantum cohomology ring does not split i.e. itself is a field, we have a stronger rigidity result as follows.

\begin{prop}\label{finite gamma}
Let $(M,\omega)$ be a closed symplectic manifold for which the spectral pseudo-norm is bounded i.e.
$$\sup \{ \gamma(H) :H\in C^{\infty}(\mathbb{R}/\mathbb{Z}\times M,\mathbb{R})\}< +\infty.$$
Let $L_1$ be a Lagrangian submanifold such that
$$HF((L_1,b_1);\Lambda)\neq 0$$
for some bounding cochain $b_1$. Then, any Lagrangian submanifold $L_2$ which is disjoint from $L_1$ has a vanishing Floer cohomology:
$$HF((L_2,b_2);\Lambda)=0$$
for any bounding cochain $b_2$.
\end{prop}

\begin{remark}\label{remark after finite gamma}
\begin{enumerate}

\item When $(M,\omega)$ is monotone, if $QH^\ast(M;\mathbb{C})$ is a field, then the spectral norm is bounded. Thus, Proposition \ref{finite gamma} applies to $\mathbb{C}P^n$ (see \cite{[EP03]}). 

\item Proposition \ref{finite gamma} is not restricted to monotone symplectic manifolds. Examples of non-monotone symplectic manifolds for which the spectral norm is bounded includes a large one point blow-up of $\mathbb{C}P^2$ and $(S^2\times S^2,\sigma \oplus \lambda \sigma)$ for $\lambda>1$ where $\sigma$ denotes an area form with $\int_{S^2} \sigma=1$. See Section \ref{lag int} for further remarks.

\end{enumerate}
\end{remark}

\begin{proof}[Proof of Theorem \ref{Q4 intersection}]

We assume that there exists a bounding cochain $b$ such that
$$HF((L,b);\Lambda)\neq 0$$
and show that $L$ must intersect either $L_0$ or $L_1$. As the closed-open map maps the identity element of the quantum cohomology ring to the identity element of the Lagrangian Floer cohomology group, we have 
$$ \mathcal{CO}^0 _b(1)=\PD([L]) \neq 0 \in HF^\ast ((L,b);\Lambda).$$
Since 
$$QH^\ast(Q^n;\mathbb{C}):=Q_+ \oplus Q_-$$
$$1=e_+ + e_-,$$
Theorem \ref{co map} implies that $L$ is either $\iota(e_+)$-heavy or $\iota(e_-)$-heavy where
$$\iota:QH^\ast(Q^n;\mathbb{C}) \hookrightarrow QH^\ast(Q^n;\Lambda).$$ 
Thus, by Corollary \ref{h to sh}, $L$ is either $e_+$-superheavy or $e_-$-superheavy (This argument was explained in more detail in the proof of Theorem \ref{main theo}.)
If $L$ intersects neither of $L_0,\ L_1$, then we have two disjoint sets which are either both $e_+$-superheavy or both $e_-$-superheavy, which contradicts Proposition \ref{disjoint-heavy}. Thus, $L$ must intersect either $L_0$ or $L_1$ and this completes the proof.
\end{proof}

We now prove Proposition \ref{finite gamma}.

\begin{proof}[Proof of Proposition \ref{finite gamma}]

Assume there exist two Lagrangian submanifolds $L_1$ and $L_2$ such that
$$L_1\cap L_2 =\emptyset$$
and 
$$HF((L_1,b_1);\Lambda) \neq 0,\ HF((L_2,b_2);\Lambda)\neq 0.$$
Then by Theorem \ref{co map}, $L_1$ and $L_2$ are both $\zeta_1$-heavy where $\zeta_1$ denotes the asymptotic spectral invariant with respect to the idempotent $1\in QH^\ast(M;\Lambda)$. Thus, for any Hamiltonian $H$ we have
$$\gamma(H)= \rho(H,1)+\rho(\overline{H},1) \geq \zeta_1 (H)+\zeta_1(\overline{H})$$
$$\geq \inf_{x\in L_1} H(x) + \inf_{x\in L_2} \overline{H}(x).$$
As $L_1\cap L_2 =\emptyset$, we can consider a Hamiltonian which is arbitrarily large on $L_1$ and arbitrarily small on $L_2$ which contradicts the assumption
$$\sup\{ \gamma(H):H\in C^{\infty}(\mathbb{R}/\mathbb{Z}\times M,\mathbb{R}) \}  < +\infty.$$
This completes the proof.
\end{proof}

As we have pointed out in Remark \ref{remark after finite gamma}, examples of closed symplectic manifolds that satisfy
$$\sup\{ \gamma(H):H\in C^{\infty}(\mathbb{R}/\mathbb{Z}\times M,\mathbb{R}) \}  < +\infty$$
include $\mathbb{C}P^n$, a large one point blow-up of $\mathbb{C}P^2$ and $(S^2\times S^2,\sigma \oplus \lambda \sigma)$ with $\lambda>1$. We provide a brief explanation to these examples.

One can easily check that, for any closed symplectic manifold $(M,\omega)$, the condition
$$\sup\{ \gamma(H):H\in C^{\infty}(\mathbb{R}/\mathbb{Z}\times M,\mathbb{R}) \}  < +\infty$$
is equivalent to 
$$\rho(\ \cdot \ , 1): \widetilde{\Ham}(M,\omega) \to \mathbb{R}$$
being a quasimorphism where $1\in QH^\ast(M;\mathbb{C})$. When $(M,\omega)$ is monotone, then $\rho( \cdot , 1)$ is a quasimorphism when $QH^\ast(M;\mathbb{C})$ is a field. Thus, the case of $\mathbb{C}P^n$ follows. When $(M,\omega)$ is non-monotone, \cite[Theorem 1.3]{[Ost06]} or \cite[Theorem 3.1]{[EP08]} imply that $\rho(\cdot , 1)$ is a quasimorphism when ``$QH^{0}(M;\mathbb{C})$'' is a field where a different set-up of the quantum cohomology is considered. For a precise definition of this set-up, we refer to \cite{[Ost06]}, \cite{[EP06]}. As pointed out in \cite[Lemma 3.1, Remark 3.4]{[Ost06]}, ``$QH^{0}(M;\mathbb{C})$'' is a field when $(M,\omega)$ is a large one point blow-up of $\mathbb{C}P^2$ or $(S^2\times S^2,\sigma \oplus \lambda \sigma)$ with $\lambda>1$.

\subsection{Proof of Application}

In this section, we prove the following Theorem, which includes Theorem \ref{leroux conj}.

\begin{theo}\label{leroux conj gene}

Let $(M,\omega)$ be a symplectic manifold which is either symplectically aspherical or monotone with the same monotonicity constant as $Q^n,\ n=2,4$ (we also allow it to be an empty set). For any $R>0$,
$$\Ham_{\geq R}:=\{\phi \in \Ham(Q^n \times M): d_{\Hof}( \id,\phi)\geq R\}$$
has a non-empty $C^0$-interior.
\end{theo}

Theorem \ref{leroux conj gene} follows as a corollary of the following statement.

\begin{theo}\label{monotone leroux}
Let $(M,\omega)$ be a monotone symplectic manifold. Assume that 
$$ \sup\{ \gamma( \phi ): \phi \in \Ham(M) \}=+\infty .$$
For any $R>0$,
$$\Ham_{\geq R}:=\{ \phi \in \Ham(M,\omega): d_{\Hof}( \id,\phi)\geq R\}$$
has a non-empty $C^0$-interior.
\end{theo}

\begin{proof}[Proof of Theorem \ref{monotone leroux}]
From the assumption, for any $R>0$, we can find $\phi \in \Ham(M, \omega)$ such that 
$$\gamma(\phi)>R+\frac{2 \dim(M)}{N_{M}} \cdot \lambda_0$$
 where $N_{M}$ is the minimal Chern number of $M$. By Theorem \ref{my lemma}, there exists $\delta>0$ such that if $d_{C^0}( \id,\phi' )<\delta$, then 
 $$\gamma(\phi' )<\frac{2 \dim(M)}{N_{M}}\cdot \lambda_0.$$ Thus, for any $\psi \in \Ham( M, \omega)$ such that $d_{C^0}(\phi,\psi)<\delta,$ we have 
$$\gamma(\psi) \geq \gamma(\phi)-\gamma(\psi \circ \phi^{-1})> R+\frac{2 \dim(M)}{N_{M}}\cdot \lambda_0-\frac{2 \dim(M)}{N_M}\cdot \lambda_0=R$$
thus, $\psi \in \Ham_{\geq R}$. This completes the proof.
\end{proof}

Now Theorem \ref{leroux conj gene} follows immediately.

\begin{proof}[Proof of Theorem \ref{leroux conj gene}]

From the K\"unneth formula for spectral invariants (see \cite[Section 5.1]{[EP09]}), we have
$$\sup \{ \gamma( H ) : H \in C^{\infty} (S^1 \times Q^n\times M) \} \geq \sup \{ \gamma( G \oplus 0 ) : G \in C^{\infty} ( S^1 \times Q^n ) \}$$
$$ = \sup \{ \gamma( G  ) : G \in C^{\infty} ( S^1 \times Q^n ) \}  \geq \sup \{ \mu ( \phi ) :\phi \in \Ham ( Q^n ) \}=+\infty$$
for $n=2,4$. Note that the last equality uses the non-triviality and the homogeneity from Theorem \ref{main theo}. The following Claim \ref{claim} implies that the above estimate is equivalent to
$$\sup \{ \gamma( \phi ) : \phi \in \Ham ( Q^n\times M) \}= +\infty$$
which completes the proof of Theorem \ref{leroux conj gene} via Theorem \ref{monotone leroux}.

\end{proof}

\begin{claim}\label{claim}
Let $(M,\omega)$ be a closed monotone symplectic manifold. For any Hamiltonian $H$, we have
$$\gamma (H)- \frac{\dim (M)}{N_M} \cdot \lambda_0 \leq \gamma (\phi_H) \leq \gamma (H ) .$$
\end{claim}

\begin{proof}[Proof of Claim \ref{claim}]
It is an easy consequence of the estimate concerning the spectral norm of Hamiltonian loops in Theorem \ref{my lemma}.
\end{proof}

\appendix

\end{document}